\documentclass[12pt,letterpaper]{amsart}
\usepackage{euler, epic,eepic,latexsym, amssymb, amscd, amsfonts, xypic, url, color, epsfig}
\input xy
\xyoption{all}

 \newlength{\baseunit}               % the basic unit length
 \newcount{\numlines}                % depth of picture (in number of lines)
\newtheorem*{tmnl}{Main Theorem}
\newtheorem{tm}{Theorem}
\newtheorem{pr}[tm]{Proposition}
\newtheorem{lm}[tm]{Lemma}
\newtheorem{co}[tm]{Corollary}

\theoremstyle{definition}
\newtheorem{df}[tm]{Definition}

\theoremstyle{remark}
\newtheorem{rmk}[tm]{Remark}
\newtheorem{expl}[tm]{Example}

\newtheorem{assumption}[tm]{Assumption}

\newcommand{\hidden}[1]{\footnote{Hidden:  #1}}
\renewcommand{\hidden}[1]{}

\newcommand{\bbA}{\mathbf{A}}

\newcommand{\bbC}{\mathbf{C}}

\newcommand{\bbF}{\mathbf{F}}

\newcommand{\bbP}{\mathbf{P}}
\newcommand{\bbQ}{\mathbf{Q}}
\newcommand{\bbR}{\mathbf{R}}

\newcommand{\bbZ}{\mathbf{Z}}

\newcommand{\calH}{{ \mathcal H}}

\newcommand{\calO}{{ \mathcal O}}

\newcommand{\Hom}{\operatorname{Hom}}
\newcommand{\Spec}{\operatorname{Spec}}
\newcommand{\ShHom}{\mathscr{H}\kern -.5pt om}
\newcommand{\cat}[1]{{\mathbf{#1}}}
\newcommand{\Spt}{\cat{Spt}}

\newcommand{\Sm}{\mathrm{Sm}}
\newcommand{\Tr}{\operatorname{Tr}}
\newcommand{\GW}{\mathrm{GW}}
\newcommand{\Th}{\operatorname{Th}}
\newcommand{\Gal}{\operatorname{Gal}}

\newcommand{\rank}{\operatorname{rank}}

\newcommand{\puritytoPnbar}{r}
\newcommand{\crush}{c}

\begin{document}
\pagestyle{plain}
\title{the class of Eisenbud--Khimshiashvili--Levine is the local $\bbA^1$-Brouwer degree}

\author{Jesse Leo Kass}

\address{Current: J.~L.~Kass, Dept.~of Mathematics, University of South Carolina, 1523 Greene Street, Columbia, SC 29208, United States of America}
\email{kassj@math.sc.edu}
\urladdr{http://people.math.sc.edu/kassj/}

\author{Kirsten Wickelgren}

\address{Current: K.~Wickelgren, School of Mathematics, Georgia Institute of Technology, 686 Cherry Street, Atlanta, GA 30332-0160}
\email{kwickelgren3@math.gatech.edu}
\urladdr{http://people.math.gatech.edu/~kwickelgren3/}


\date{\today}

\maketitle

\begin{abstract}
Given a polynomial function with an isolated zero at the origin, we prove that the local  $\bbA^{1}$-Brouwer degree equals the Eisenbud--Khimshiashvili--Levine class.  This answers a question posed by David Eisenbud in 1978. We give an application to counting nodes together with associated arithmetic information by enriching Milnor's equality between the local degree of the gradient and the number of nodes into which a hypersurface singularity bifurcates to an equality in the Grothendieck--Witt group. %We then apply this result to count singularities.
\end{abstract}

{\parskip=12pt % closing bracket is just before the bibliography blah

We prove that the Eisenbud--Khimshiashvili--Levine class of a polynomial function with an isolated zero at the origin is the local $\bbA^1$-Brouwer  degree, a result that answers a question of Eisenbud.

The classical local Brouwer degree $\deg_{0}(f)$ of a continuous function $f \colon \bbR^{n} \to \bbR^{n}$ with an isolated zero at the origin is  the image  $\deg(f/|f|) \in \bbZ$ of the map  of $(n-1)$-spheres
\[
	f/|f| \colon S_{\epsilon}^{n-1} \to S^{n-1}, \epsilon>0 \text{ sufficiently small,}
\]
under the global Brouwer degree homomorphism $\deg \colon [S^{n-1}_{\epsilon}, S^{n-1}] \to \bbZ$.

When $f$ is a $C^{\infty}$ function, Eisenbud--Levine and independently Khimshiashvili constructed a real nondegenerate symmetric bilinear form (more precisely, an isomorphism class of such forms) $w_{0}(f)$ on the local algebra $Q_{0}(f) := C^{\infty}_{0}(\mathbb{R}^{n})/(f)$ and proved
\begin{equation} %\label{Eqn: SigFormula}
	\deg_{0}(f) = \text{ the signature of }w_{0}(f)
\end{equation}
(\cite[Theorem~1.2]{eisenbud77}, \cite{khimshiashvili}; see also \cite[Chapter~5]{arnold12} and \cite{Khimshiashvili01}). If we further assume that $f$ is real analytic, then we can form the complexification $f_{\bbC} \colon \bbC^{n} \to \bbC^{n}$, and  Palamodov \cite[Corollary~4]{palamodov} proved an analogous result for $f_{\bbC}$:
\begin{equation} %\label{Eqn: RankFormula}
	\deg_{0}(f_{\bbC}) = \text{ the rank of }w_{0}(f).
\end{equation}
Eisenbud observed that the definition of $w_{0}(f)$ remains valid when $f$ is a polynomial with coefficients in an arbitrary  field $k$ and asked whether this form can be identified with a degree in algebraic topology \cite[Some remaining questions (3)]{eisenbud78}.  Here we answer Eisenbud's question by proving that $w_{0}(f)$ is the local Brouwer degree in $\bbA^{1}$-homotopy theory. More specifically, we prove
\begin{tmnl}%[Main Theorem]
	If $f \colon \bbA_{k}^{n} \to \bbA_{k}^{n}$ has an isolated zero at the origin, then 
	\begin{equation} \label{Eqn: MainThmEqn}
		\deg^{\bbA^1}_{0}(f) = \text{the stable isomorphism class of $w_{0}(f)$.}
	\end{equation}
\end{tmnl}

Morel described the degree map in $\bbA^1$-homotopy theory in his 2006 presentation at the International Congress of Mathematicians \cite{morel06}.  In $\bbA^1$-homotopy theory, one of several objects that plays the role of the sphere is $\bbP^{n}_{k}/\bbP^{n-1}_{k}$, the quotient of $n$-dimensional projective space by the $(n-1)$-dimensional projective space at infinity.  Morel constructed a group homomorphism 
\[
	\deg^{\bbA^1} \colon [\bbP^{n}_{k}/\bbP^{n-1}_{k},\bbP^{n}_{k}/\bbP^{n-1}_{k}] \to \GW(k)
\]
from the $\bbA^1$-homotopy classes of endomorphisms of $\bbP^{n}_{k}/\bbP^{n-1}_{k}$ to the Groethendieck--Witt group, which is the groupification of the monoid of (isomorphism classes of) nondegenerate symmetric bilinear forms over $k$.  The local degree is defined in terms of the global degree in the natural manner, as we explain in Section~\ref{Section: LocalDegree}. 

The proof of the Main Theorem runs as follows.  When $f$ has a simple zero at the origin, we prove the result by directly computing that both sides of \eqref{Eqn: MainThmEqn} are represented by the class of the Jacobian $\langle \det( \frac{\partial f_i}{\partial x_j}(0)) \rangle$.  When $f$ has a simple zero at a nonrational point, we show an analogous equality using work of Hoyois \cite{Hoyois_lef}. Using the result for a simple zero, we then prove the result when $f$ has an arbitrary zero.  We begin by reducing to the case where $f$ is the restriction of a morphism $F \colon \bbP^{n}_{k} \to \bbP^{n}_{k}$ satisfying certain technical conditions (those in Assumption~\ref{Assumption}) that include the condition that all zeros of $F$ other than the origin are simple.  For every closed point $x \in \bbA^{n}_{k}$, Scheja--Storch have constructed a bilinear form whose class $w_{x}(F)$ equals the Eisenbud--Khimshiashvili--Levine class when $x=0$.  From the result on simple zeros, we deduce that 
\begin{equation} \label{Eqn: PfSketchStep}
	\sum_{x \in f^{-1}(\overline{y})} \deg_{x}(F) = \sum_{x \in f^{-1}(\overline{y})} w_{x}(F) 
\end{equation}
holds for $\overline{y} \in \bbA^{n}_{k}(k)$ a regular value.    For $\overline{y}$ arbitrary (and possibly not a regular value), we show that both sums in \eqref{Eqn: PfSketchStep} are independent of $\overline{y}$, allowing us to conclude that \eqref{Eqn: PfSketchStep} holds for all $\overline{y}$.  In particular, equality holds for $\overline{y}=0$.  For $\overline{y}=0$, we have $\deg_{x}(F) = w_{x}(F)$ for $x \in f^{-1}(0)$ not equal to the origin by the result for simple zeros, and taking differences, we deduce the equality for $\deg_{0}(F) = w_{0}(F)$, which is the Main Theorem.

We propose counting singularities arithmetically, and in Section~\ref{Section: SingThy}, we do so using the Main Theorem in the manner that we now describe.  Suppose $\operatorname{char} k \ne 2$ and $n$ is even, and let $f \in k[x_1, \dots, x_n]$ be the equation of an isolated hypersurface singularity 
\[
	0 \in X := \{ f=0 \}  \subset \bbA^{n}_{k}
\]
at the origin.  We define the arithmetic (or $\bbA^1-$) Milnor number by $\mu^{\bbA^1}(f) := \deg^{\bbA^1}_{0}(\operatorname{grad}(f))$ and show that this invariant is an arithmetic count of the nodes (or $A_1$-singularities) to which $X$ bifurcates. Suppose $\operatorname{grad}(f)$ is finite and separable.  Then for general $(a_1, \dots, a_n) \in \bbA^{n}_{k}(k)$, the family 
\begin{equation} \label{Eqn: FirstDeformedFamily}
	f(x_1, \dots, x_n) + a_1 x_1 + \dots + a_n x_n = t
\end{equation}
over the affine $t$-line contains only nodal fibers as singular fibers.   For simplicity, assume that the origin is the only zero of $\operatorname{grad} f$ and the nodes appearing in \eqref{Eqn: FirstDeformedFamily} all have residue field $k$ (rather than a nontrivial extension). We then have
\begin{equation} \label{Eqn: XBifurcatesToNodes}
	\mu^{\bbA^1}(f) = \sum \#(\text{nodes with henselization $\{ u_1 x_1^2+\dots+u_n x_n^2=0\}$}) \cdot \langle u_1 \dots u_n \rangle \text{ in $\operatorname{GW}(k)$.}
\end{equation}
Here the sum runs over isomorphic classes of henselizations of rings $k[x_1, \dots, x_n] / u_1 x_1^2+\dots +u_n x_n^2$, and the naive count is of nodal fibers of  \eqref{Eqn: FirstDeformedFamily}.

Taking the rank of both sides of Equation~\eqref{Eqn: XBifurcatesToNodes}, we deduce that the number of nodal fibers equals the rank of $\mu^{\bbA^1}(f)$.  When $k=\bbC$, this fact was observed by Milnor \cite[page~113, Remark]{milnor68}, and \eqref{Eqn: XBifurcatesToNodes} should be viewed as an enrichment of Milnor's result from an equality of integers to an equality of classes in $\operatorname{GW}(k)$. When $k=\bbR$, the real realization of Equation~\eqref{Eqn: XBifurcatesToNodes} was essentially proven by Wall \cite[page 347, esp.~second displayed equation]{wall83}.

Through Equation~\eqref{Eqn: XBifurcatesToNodes}, the arithmetic  Milnor number provides a computable constraint on the nodes to which a hypersurface singularity can bifurcate.  As an illustration, consider the cusp (or $A_2$-singularity) $\{ x_1^2 + x_2^3=0\}$ over the field $\bbQ_p$ of $p$-adic numbers.  A computation shows that  $\mu^{\bbA^1}(f)$ has rank $2$ and discriminant $-1 \in \bbQ_{p}^{\ast}/(\bbQ_{p}^{\ast})^{2}$.  When $p=5$, we have $-1 \neq 1 \cdot 2$ in $\bbQ_{p}^{\ast}/(\bbQ_{p}^{\ast})^{2}$, so we conclude that  the cusp cannot  bifurcate to the split node $\{ x_1^2+ x_2^2=0\}$ and the nonsplit node $\{ x_1^2+ 2 \cdot  x_2=0\}$.  When $p=11$, $\mu^{\bbA^1}(f)$ does not provide such an obstruction and in fact, those two nodes are the singulars fibers of $x_1^2 + x_2^3 + 10 \cdot x_2=t$.  We discuss this example in more detail towards the end of Section~\ref{Section: SingThy}.

Arithmetic Milnor numbers, and other local $\mathbb{A}^{1}$-degrees,  also appear in enumerative results of M.~Levine.  (Note: this person is different from the second author of \cite{eisenbud77}).  Indeed, in  \cite{levine17} Levine establishes a formula giving an enumerative count of the singular fibers of a suitable fibration of a smooth projective variety over a curve in which an isolated singularity of a fiber is weighted by $\mu^{\bbA^1}(f)$.  Levine's formula should be viewed as a global analogue of \eqref{Eqn: XBifurcatesToNodes}.  Levine also computes $\mu^{\mathbb{A}^{1}}(f)$ when $f$ satisfies a certain ``diagonizability"  hypothesis.     A different application of the $\mathbb{A}^{1}$-degree to enumerative geometry is given by the present authors in \cite{kass17}.  There the authors study a weighted count of the lines on a cubic surface with the weights defined as local $\mathbb{A}^{1}$-degrees.  We discuss the application to cubic surfaces in Section~\ref{Section: CubicSurface}.

The results of this paper are related to results in the literature.  We have already discussed the work of Eisenbud--Khimshiashvili--Levine and Palamodov describing $w_{0}(f)$ when $k=\bbR, \bbC$.  When $k$ is an ordered field, B{\"o}ttger--Storch studied the properties of $w_{0}(f)$ in \cite{boettger}.  They defined the mapping degree of $f \colon \bbA^n_{k} \to \bbA^{n}_{k}$ to be the signature of $w_{0}(f)$  \cite[4.2~Definition, 4.3~Remark]{boettger} and then proved that the mapping degree is a signed count of the points in the preimage of a regular value \cite[Theorem~4.5]{boettger}.

Grigor{$'$}ev--Ivanov studied $w_{0}(f)$ when  $k$ is an arbitrary field in  \cite{ivanov}.  They prove that a sum of these classes in a certain quotient of the Grothendieck--Witt group is a well-defined invariant of a rank $n$ vector bundle on a suitable $n$-dimensional smooth projective variety \cite[Theorem~2]{ivanov}.  (This invariant should be viewed as an analogue of the Euler number. Recall that, on an oriented $n$-dimensional manifold, the Euler number of an oriented rank $n$ vector bundle can be expressed as a sum of the local Brouwer degrees associated to a general global section.)   

The Main Theorem is also related to Cazanave's work on the global $\bbA^1$-degree of a rational function.  In \cite{cazanave08, cazanave}, Cazanave proved that the global $\bbA^1$-degree of a rational function $F \colon \bbP^{1}_{k} \to \bbP^{1}_{k}$ is the class represented by the B\'{e}zout matrix, an explicit symmetric matrix.  The class $w_{0}(f)$ is a local contribution to the class of the B\'{e}zout matrix because the global degree is a sum of local degrees, so it is natural to expect the B\'{e}zout matrix to be directly related to a bilinear form on $Q_{0}(f)$.   As we explain in the companion paper \cite{wickelgren16}, such a direct relation holds:  the B\'{e}zout matrix is the Gram matrix of the residue form,  a symmetric bilinear form with orthogonal summand representing  $w_{0}(f)$.  

\section*{Conventions}
$k$ denotes a fixed field.

We write $P$ or $P_{x}$ for the polynomial ring $k[x_1, \dots, x_n]$ and $\mathfrak{m}_{0}$ for the ideal $(x_1, \dots, x_n) \subset P$.  We write $P_{y}$ for $k[y_1, \dots, y_n]$.  We write $\widetilde{P}$ for the graded ring $k[X_0, \dots, X_n]$ with grading $\operatorname{deg}(X_i)=1$.  We then have $\bbP_{k}^{n} = \operatorname{Proj} \widetilde{P}$.

If $f \colon \bbA^{n}_{k} \to \bbA^{n}_{k}$ is a polynomial function, then we write $f_1, \dots, f_n \in P_{x}$ for the components of $f$.  We say a polynomial function $f \colon \bbA^{n}_{k} \to \bbA^{n}_{k}$ has an isolated zero at a closed point $x \in \bbA^{n}_{k}$  if the local algebra $Q_{x}(f) := P_{\mathfrak{m}_{x}}/(f_1, \dots, f_n)$ has finite length.    We say that a closed point $x$ of $\bbA^n_k$ is isolated in its fiber $f^{-1}(f(x))$ if $x$ is a connected component $f^{-1}(f(x))$, or equivalently, if there is a Zariski open neighborhood of $U$ of $x$ in $\bbA^n_k$ such that $f$ maps $U -\{x\}$ to $\bbA^n_k - \{f(x) \}$. Note that if $f$ has an isolated zero at the origin, then $Q_{0}(f)$ has dimension $0$, which implies that the connected component of $f^{-1}(0) \cong \Spec P/(f_1, \dots, f_n)$ containing $0$ contains no other points, whence $0$ is isolated in its fiber. 

Using homogeneous coordinates $[X_0, X_1, \ldots, X_n]$ for $\bbP^n_{k}$, we use $\bbA^n_k$ to denote the open subscheme of $\bbP^n_{k}$ where $X_0 \neq 0$, and $\bbP^{n-1}_{k}$ to denote its closed complement isomorphic to $\bbP^{n-1}_{k}$.

For a vector bundle $\mathcal{E}$ on a smooth scheme $X$, let $\Th(\mathcal{E})$ denote the Thom space of $\mathcal{E}$ of Section 3, Definition 2.16 of \cite{morelvoevodsky1998}, i.e., $\Th(\mathcal{E})$ is the pointed sheaf $$\Th(\mathcal{E}) = \mathcal{E}/(\mathcal{E} - z(X)),$$ where $z: X \to \mathcal{E}$ denotes the zero section.

It will be convenient to work in the stable $\bbA^1$-homotopy category $\Spt(B)$ of $\bbP^1$-spectra over $B$, where $B$ is a finite type scheme over $k$. Most frequently, $B=L$, where $L$ is a field extension of $k$. The notation $[-,-]_{\Spt(B)}$ will be used for the morphisms. $\Spt(B)$ is a symmetric monoidal category under the smash product $\wedge$, with unit $1_B$, denoting the sphere spectrum. Any pointed simplicial presheaf $X$ determines a corresponding $\bbP^1$-suspension spectrum $\Sigma^{\infty} X$. For example, $\Sigma^{\infty} \Spec L_+ \cong 1_L$ and $\Sigma^{\infty} (\bbP^1_L)^{\wedge n}$ is a suspension of $1_L$. When working in $\Spt(L)$, we will identify pointed spaces $X$ with their suspension spectra $\Sigma^{\infty} X$, omitting the $\Sigma^{\infty}$. We will use the six operations $(p^*, p_*, p_!, p^!, \wedge, \Hom)$ given by Ayoub \cite{Ayoub_sixop1} and developed by Ayoub, and Cisinksi-D\'eglise \cite{CD-triang_cat_mixed_motives}. There is a nice summary in \cite[\S 2]{Hoyois_lef}. We use the following associated notation and constructions. When $p: X \to Y$ is smooth, $p^*$ admits a left adjoint, denoted $p_{\sharp}$, induced by the forgetful functor $\Sm_{X} \to \Sm_{Y}$ from smooth schemes over $X$ to smooth schemes over $Y$.  For $p:X\to \Spec L$ a smooth scheme over $L$, the suspension spectrum of $X$ is canonically identified with $p_!p^!1_L$ as an object of $\Spt(L)$. For a vector bundle $p:E \to X$, the Thom spectrum $\Sigma^{\infty} \Th(E)$ (or just $\Th(E)$) is canonically identified with $s^*p^! 1_X$. Let $\Sigma^E$ equal $\Sigma^E = s^* p^!: \Spt(X) \to \Spt(X)$. Let $e: E \to X$ and $d: D \to Y$ be two vector bundles over smooth $L$-schemes $p: X \to \Spec L$ and $q:Y \to \Spec L$. Given a map $f: Y \to X$ and a monomorphism $\phi: D \hookrightarrow f^* E$, there is an associated natural transformation $$\Th_f \phi : q_! \Sigma^D q^! \to p_! \Sigma^E p^!$$ of endofunctors on $\Spt(L)$ inducing the map on Thom spectra. The natural transformation $\Th_f \phi $ is defined as the composition \begin{equation}\label{thfphi=comp}\Th_f \phi = \Th_f {1_{f^*E}} \circ \Th_{1_Y} \phi.\end{equation} The natural transformation $ \Th_{1_Y} \phi$ is the composition $$ t^*d^! \cong t^* \phi^!e^!\to t^* \phi^* e^! \cong s^* e^!,$$ where $t: Y \to D$ denotes the zero section of $D$, $s: X \to E$ denotes the zero section of $E$, and the middle arrow is induced by the exchange transformation $ \phi^! \cong 1^* \phi!* \to 1^! \phi^* \cong \phi^*$.  The natural transformation $\Th_f {1_{f^*E}}$ is the composition \begin{equation}\label{Thf1_as_counit} \Th_f 1: q_! \Sigma^{f^* E} q^! \cong p_! f_! \Sigma^{f^*E} f^! p^! \cong p_!\Sigma^E f_! f^! p^! \stackrel{\epsilon}{\rightarrow} p_! \Sigma^E p^!,\end{equation} where $\epsilon: f_! f^! \to 1$ denotes the counit.

\section{The Grothendieck--Witt class of Eisenbud--Khimshiashvili--Levine} \label{Section: LocalForm}
In this section we recall the definition of the Grothendieck--Witt class  $w_{0}(f)$ studied by Eisenbud--Khimshiashvili--Levine.  We compute the class when $f$ has a nondegenerate zero and when $f$ is the gradient of the equation of an ADE singularity.   Here $f \colon \bbA^{n}_{k} \to \bbA^{n}_{k}$ is a polynomial function with an isolated zero at the origin (i.e.~$0$ is a connected component of $f^{-1}(0)$).  We write $f_1, \dots, f_n \in  P$ for the components of $f$.  

\begin{df} \label{Definition: LocalAlgbra}
	Suppose that $x \in \bbA^{n}_{k}$ is a closed point such that $y=f(x)$ has residue field $k$.  Writing the maximal ideal of $x$ and $y$ respectively  as $\mathfrak{m}_{x}$ and  $\mathfrak{m}_{y}= (y_1-\overline{b}_{1}, \dots, y_n-\overline{b}_{n})$, we define the  \textbf{local algebra} $Q_{x}(f)$ of $f$ at a closed point $x$ to be  $P_{\mathfrak{m}_{x}}/(f_1-\overline{b}_{1}, \dots, f_n-\overline{b}_{n})$.  When $x=0$, we also write $Q$ for $Q_{0}(f)$, the local algebra at the origin.

	The \textbf{distinguished socle element}  at the origin  $E=E_{0}(f) \in Q_{0}(f)$ is
	\[
		E_{0}(f) := \det \begin{pmatrix} a_{i, j} \end{pmatrix}
	\]
	for $a_{i, j} \in P$ polynomials satisfying
	\[
		f_{i}(x) = f_{i}(0)+\sum_{j=1}^{n} a_{i, j} x_{j}.
	\]
	The \textbf{Jacobian element} at the origin $J=J_{0}(f) \in Q_{0}(f)$ is 
	\[
		J_{0}(f) := \det( \frac{\partial f_i}{\partial x_j}).
	\]
\end{df}
\begin{rmk}
	Recall the socle of a ring is the sum of the minimal nonzero ideals.  For an artin local ring such as $Q_{0}(f)$, the socle is equal to the annihilator of the maximal ideal $\mathfrak{m}$.  We only use the definition of the socle in Lemma~\ref{Lemma: SocleElement}, which is used to prove Lemma~\ref{Lemma: WhenFormsEqual}.
\end{rmk} 

\begin{rmk}\label{rmkJ=rE}
	The elements $E$ and $J$ are related by $J = (\operatorname{rank}_{k} Q_{0}(f))  \cdot E$ by \cite[(4.7)~Korollar]{scheja}.  From this, we see that the two elements contain essentially the same information when $k=\bbR$ (the case studied in \cite{eisenbud77}), but $E$ contains more information when  the characteristic of $k$ divides the rank of $Q_{0}(f)$.
\end{rmk} 

\begin{lm} \label{Lemma: SocleElement}
If $f$ has an isolated zero at the origin, then the socle of $Q_{0}(f)$ is generated by $E$.  
\end{lm}
\begin{proof}	
	Since $Q$ is Gorenstein (by e.g.~Lemma~\ref{Lemma: RelativeCompleteIntersection}) with residue field $k$, the socle is $1$-dimensional, so it is enough to prove that $E$ is nonzero and in the socle.  This follows from the proof of  \cite[(4.7)~Korolllar]{scheja}.  In the proof, Scheja--Storch show  that $E = \Theta(\pi)$ for $\pi \colon Q \to k$ the evaluation function $\pi(a) = a(0)$ and $\Theta \colon \operatorname{Hom}_{k}(Q, k) \cong Q$ a certain isomorphism of $Q$-modules (for the $Q$-module structure on $\operatorname{Hom}_{k}(Q, k)$ defined by $(a \cdot \phi)(b)=\phi(a \cdot b)$).  The maximal ideal $\mathfrak{m}_{0}=(x_{1}, \dots, x_{n})$ of $Q$ must annihilate $E$ since the ideal annihilates $\pi$, and so $E$ lies in the socle.  Furthermore, $E$ is nonzero since $\pi$ is nonzero. 
\end{proof}

\begin{df}
	If  $\phi \colon Q_{0}(f) \to k$ is a $k$-linear function, then we define a symmetric bilinear form  $\beta_{\phi} \colon Q \times Q \to k$ by $\beta_{\phi}(a_1, a_2) := \phi(a_1 \cdot a_2)$.
\end{df}

\begin{lm} \label{Lemma: WhenFormsEqual}
	If $\phi_1$ and $\phi_2$ are $k$-linear functions satisfying $\phi_1(E) = \phi_2(E) \text{ in $k/(k^{\ast})^{2}$}$, then $\beta_{\phi_{1}}$ is isomorphic to $\beta_{\phi_{2}}$.  Furthermore, if $\phi(E) \ne 0$, then  $\beta_{\phi}$ is nondegenerate.
\end{lm}
\begin{proof}
	Since $E$ generates the socle, the result follows from \cite[Propositions~3.4, 3.5]{eisenbud77}. 
\end{proof}

Refer to the Erratum at the end of this text for a necessary modification of Lemma \ref{Lemma: WhenFormsEqual}.

\begin{df} \label{Definition: EKLform}
	 The Grothendieck--Witt  \textbf{class of Eisenbud--Khimshiashvili--Levine} or the \textbf{EKL class} $w = w_{0}(f) \in \operatorname{GW}(k)$ is the Grothendieck--Witt class of $\beta_{\phi}$ for any $k$-linear function $\phi \colon Q \to k$ satisfying $\phi(E) = 1$.
\end{df}
Recall that the Grothendieck--Witt group  $\operatorname{GW}(k)$ of $k$ is the groupification of the monoid of nondegenerate symmetric bilinear forms \cite[Definition~1.1]{lam05}. The Grothendieck--Witt class $w_{0}(f)$ is independent of the choice of $\phi$ by Lemma~\ref{Lemma: WhenFormsEqual}.

In this paper we focus on the class $w_{0}(f)$, but in work recalled in Section~\ref{Section: Family},  Scheja--Storch  constructed a distinguished symmetric bilinear form $\beta_0$ that represents $w_{0}(f)$.  This symmetric bilinear form encodes more  information than $w_{0}(f)$ when $f$ is a polynomial in $1$ variable, and we discuss this topic in greater detail in \cite[Section~4]{wickelgren16}.  

To conclude this section, we explicitly describe some ELK classes.    The descriptions are in terms of the following classes.
\begin{df}
	Given $\alpha_1, \dots, \alpha_{m} \in k^{\ast}$, we define  $\langle \alpha_{1}, \dots, \alpha_{m} \rangle \in \operatorname{GW}(k)$ to be the class of the symmetric bilinear form 
	\begin{gather*}
		\beta  \colon k^{\oplus m}  \times k^{\oplus m} \to k,\\
		\beta((a_{1}, \dots, a_{m}), (b_1, \dots, b_m)) = \alpha_{1} \cdot a_{1} b_{1} + \dots + \alpha_{m} \cdot a_{m} b_{m}.
	\end{gather*}
	The standard hyperbolic form $\mathbb{H}$ is the symmetric bilinear form 	
		\begin{gather*}
		\beta  \colon k^{\oplus 2}  \times k^{\oplus 2} \to k,\\
		\beta((a_{1}, a_{2}), (b_1, b_2)) = a_{1} b_{2} +  a_{2} b_{1}.
	\end{gather*}
	The class of $\mathbb{H}$ equals $\langle 1, -1 \rangle$ in $\operatorname{GW}(k)$.
\end{df}

The following lemma describes  $w_{0}(f)$ when $f$ has  a simple zero.
\begin{lm}	
	If $f$ has a simple zero at the origin, then $w_{0}(f) = \left\langle \det \frac{\partial f_i}{\partial x_j}(0) \right\rangle$.
\end{lm}
\begin{proof}
	We have $Q_{0}(f) = k$ and $E = \det \frac{\partial f_i}{\partial x_j}(0)$.  The element $E \in Q_{0}(f)$ is then a $k$-basis, and $w_{0}(f)$ is represented by the form $\beta_{\phi}$ satisfying 
	\begin{align*}
		\beta_{\phi}(E, E) =& \phi( \det \frac{\partial f_i}{\partial x_j}(0) \cdot E) \\
					=& \det \frac{\partial f_i}{\partial x_j}(0).
	\end{align*}
\end{proof}
	
	When $f$ has an arbitrary isolated zero, the following procedure computes $w_{0}(f)$.
	
\begin{table}[htp] \label{Table: HowToCompute}
\caption{Method for computing the ELK class of $f$}
\begin{center}
\begin{enumerate}
	\item Compute a Gr\"{o}bner (or standard) basis for the ideal $(f_1,\dots, f_n)$ and a $k$-basis for the vector space $Q_{0}(f)$.
	\item Express $E$ in terms of the $k$-basis by performing a division with the Gr\"{o}bner basis.
	\item Define an explicit $k$-linear function $\phi \colon Q_{0}(f) \to k$ satisfying $\phi(E)=1$ using the $k$-basis.
	\item For every pair $b_i, b_j$ of basis elements, express $b_i \cdot b_j$ in terms of the $k$-basis by performing a division and then use that expression to evaluate $\phi( b_{i} \cdot b_{j})$.  
	\item Output: The matrix with entries $\phi( b_{i} \cdot b_{j} )$ is the Gram matrix of a symmetric bilinear form that represents $w_{0}(f)$.
\end{enumerate}
\end{center}
\end{table}
(For a detailed exposition on how to compute in a finite dimensional $k$-algebra such as $Q_{0}(f)$, see Section~2, Chapter~2 and Chapter~4 of \cite{cox05}.)

 Table~\ref{Table: ADEsing} describes some classes that were computed using this procedure.  The table should be read as follows.  The second column displays a polynomial $g$, namely the polynomial equation of the ADE singularity named in the first column.  The associated gradient $\operatorname{grad}(g) := (\frac{\partial g}{\partial x}, \frac{\partial g}{\partial y})$ is a polynomial function $\bbA^2_{\bbQ} \to \bbA^{2}_{\bbQ}$ with an isolated zero at the origin, and the third column is its ELK class $w_{0}(\operatorname{grad}(g)) \in \operatorname{GW}(\bbQ)$.  (We consider $g$ as a polynomial with rational coefficients.)

\begin{table}[htp] 
\caption{ELK classes for ADE singularities}  \label{Table: ADEsing}
\begin{center}
\begin{tabular}{l l l}
	Singularity 			&	Equation $g$					&	 $w_{0}(\operatorname{grad}(g)) \in \operatorname{GW}(\bbQ)$ \\
	\hline \hline	 													 \\
	$A_{n}$, $n$ odd		&	$x_1^2+x_2^{n+1}$				&	$ \frac{n-1}{2} \cdot \mathbb{H}+\left \langle 2(n+1) \right \rangle$							\\
	$A_{n}$, $n$ even		&	$x_1^2+x_2^{n+1}$				&	$\frac{n}{2} \cdot \mathbb{H}$														\\
	$D_{n}$, $n$ even		&	$x_{2} (x_{1}^{2}+x_{2}^{n-2})$		&	$\frac{n-2}{2} \cdot \mathbb{H}+\left\langle  -2, 2(n-1)  \right\rangle$							\\
	$D_{n}$, $n$ odd		&	$x_{2} (x_{1}^{2}+x_{2}^{n-2})$		&	$\frac{n-1}{2} \cdot \mathbb{H}+\langle -2 \rangle$ 										\\
	$E_{6}$				& 	$x_{1}^{3}+x_{2}^{4}$			&	$3 \cdot \mathbb{H}$															\\
	$E_{7}$				& 	$x_{1} (x_{1}^{2}+x_{2}^{3})$		&	$3 \cdot \mathbb{H}+\left\langle -3 \right\rangle $										\\
	$E_{8}$				&	$x_{1}^{3}+x_{2}^{5}$				&	$4 \cdot \mathbb{H}$								
\end{tabular}
\end{center}
\end{table}

The description of $w_{0}(\operatorname{grad}(g))$ in Table~\ref{Table: ADEsing} remains valid when $\bbQ$ is replaced by field of characteristic $0$ or $p>0$ for $p$ sufficiently large relative to $n$ but possibly not for small $p$ (e.g.~the description of the $A_2$ singularity is invalid in characteristic $3$ because $\operatorname{grad}(g)$ has a nonisolated zero at the origin).

\section{Local $\bbA^1$-Brouwer degree} \label{Section: LocalDegree}

Morel's $\bbA^1$-Brouwer degree homomorphism \begin{equation*}
\deg: [(\bbP_k^1)^{\wedge n},(\bbP_k^1)^{\wedge n}] \to \GW(k) 
\end{equation*} gives rise to a notion of local degree, which we describe in this section. We then show that the degree is the sum of local degrees under appropriate hypotheses (Proposition \ref{pr_deg_is_sum_local_deg}), and that when $f$ is \'etale at $x$, the local degree is computed by $\deg^{\bbA^1}_x f =\Tr_{k(x)/k} \langle J (x) \rangle$, where $J(x)$ denotes the Jacobian determinant $J = \det \begin{pmatrix} \frac{\partial f_i}{\partial x_j} \end{pmatrix}$ evaluated at $x$ (Proposition~\ref{loc_degree_etale_point}). For endomorphisms of $\bbP^1_k$, these notions and properties are stated in \cite{Morel_motivicpi0_sphere} \cite{morel06}, and build on ideas of Lannes. To identify the local degree at an \'etale point, we use results of Hoyois \cite{Hoyois_lef}.

To motivate the definition, recall that to define the local topological Brouwer degree of $f: \bbR^n \to \bbR^n$ at a point $x$, one can choose a sufficiently small $\epsilon>0$ and take the $\bbZ$-valued topological degree of the map $$\xymatrix{S^{n-1} \cong \{ z : \| z - x \| = \epsilon \} \ar[rrr]^{\frac{f - f(x)}{\| f - f(x) \|}} &&&  \{ z : \| z  \| =1 \} \cong S^{n-1} }.$$ By translation and scaling, the map $\frac{f - f(x)}{\| f - f(x) \|}$ can be replaced by the map induced by $f$ from the boundary $\partial B(x, \epsilon)$ of a small ball $B(x, \epsilon)$ centered at $x$ to a boundary $\partial B(f(x), \epsilon')$ of a small ball centered at $f(x)$. The suspension of this map can be identified with the map induced by $f$ \begin{equation}\label{loc_deg_ball_relative_map}f: \frac{B(x, \epsilon)}{\partial B(x, \epsilon)} \to \frac{B(f(x), \epsilon')}{\partial B(f(x), \epsilon')},\end{equation} from the homotopy cofiber of the inclusion $\partial B(x, \epsilon) \to B(x, \epsilon) $ to the analogous homotopy cofiber. As $\frac{B(x, \epsilon)}{\partial B(x, \epsilon)}$ is also the homotopy cofiber of $B(x, \epsilon) -\{x\}\to B(x, \epsilon) $, we are free to use the latter construction for the (co)domain in \eqref{loc_deg_ball_relative_map}: \begin{equation}\label{loc_deg_ball_rel2}f: \frac{B(x, \epsilon)}{B(x, \epsilon)-\{x\}} \to \frac{B(f(x), \epsilon')}{B(f(x), \epsilon') - \{f(x)\}}.\end{equation} 

In $\bbA^1$-algebraic topology, the absence of small balls around points whose boundaries are spheres makes the definition of local degree using the map $\frac{f - f(x)}{\| f - f(x) \|}$ problematic. However, the map \eqref{loc_deg_ball_rel2} generalizes to a map between spheres by Morel and Voevodsky's Purity Theorem. This allows us to define a local degree when $x$ and $f(x)$ are both rational points, as in the definition of $f_x'$ given below. When $x$ is not rational, we precompose with the collapse map from the sphere $\bbP^n_k / \bbP_{k}^{n-1} \to \bbP^n_k /\bbP^n_k -\{x\}$ to obtain Definition \ref{def:local_degree_f(x)_rational}. This is shown to be compatible with the former definition (Proposition \ref{local_degree_alternate_def}).  

We now give Definition \ref{def:local_degree_f(x)_rational}, first introducing the necessary notation.

By \cite[Proposition~2.17 numbers 1 and 3, page~112]{morelvoevodsky1998}, there is a canonical $\bbA^1$-weak equivalence $(\bbP_k^1)^{\wedge n} \cong  \bbP_k^n/\bbP_k^{n-1}$ as both can be identified with the Thom space $\Th(\mathcal{O}_k^n)$ of the trivial rank $n$ bundle on $\Spec k$. Thus we may take the degree of a map $ \bbP_k^n/\bbP_k^{n-1} \to  \bbP_k^n/\bbP_k^{n-1}$ in the homotopy category.

Let $x$ be a closed point of $\bbA^n_k$, and let  $f: \bbA^n_k \to \bbA^n_k$ be a function such that $x$ is isolated in its fiber $f^{-1}(f(x))$. Choose a Zariski open neighborhood $U$ of $x$ such that $f$ maps $U - \{x\} $ into $\bbA^n_k - \{ f(x)\}$.  The Nisnevich local homotopy push-out diagram $$\xymatrix{ U -\{x\} \ar[r] \ar[d] & \bbP_k^{n} - \{x\} \ar[d]  \\ U \ar[r] & \bbP_k^n} $$ induces a canonical homotopy equivalence $U/(U - \{ x\}) \to \bbP_k^n/\bbP_k^{n} - \{x\}$. 

There is a trivialization of $T_x \bbP^n_k$ coming from the isomorphism $T_x \bbP_k^n \cong T_x \bbA_k^n$ and the canonical trivialization of $T_x \bbA^n_k$. Purity thus induces an $\bbA^1$-weak equivalence $ \bbP_k^n/(\bbP_k^{n} - \{x\}) \cong \Th(\mathcal{O}^n_{k(x)}) $. As above, \cite[Proposition 2.17 number 3, page~12]{morelvoevodsky1998} gives a canonical $\bbA^1$-weak equivalence $\Th(\mathcal{O}^n_{k(x)}) \cong \bbP_{k(x)}^n/(\bbP_{k(x)}^{n-1})$.  Let $\puritytoPnbar:  \bbP_k^n/(\bbP_k^{n} - \{x\}) \stackrel{\cong}{\to}  \bbP_{k(x)}^n/(\bbP_{k(x)}^{n-1})$ denote the composite $\bbA^1$-weak equivalence.

For $n = 1$, the following lemma is \cite[Lemma 5.4]{Hoyois_lef}, and the proof generalizes to the case of larger $n$, the essential content being \cite[Lemma 2.2]{Voevodsky_MCZ2}.

\begin{lm}\label{crush=id}
For any $k$-point $x$ of $\bbA^n_k$, the composition $$ \crush_x: \bbP_k^n/(\bbP_k^{n-1}) \to \bbP_k^n/(\bbP_k^{n} - \{x\}) \cong \bbP_k^n/(\bbP_k^{n-1})$$ of the collapse map with $\puritytoPnbar$ is $\bbA^1$-homotopy equivalent to the identity.
 
\end{lm}

\begin{proof}
Let $[X_0, X_1, \ldots, X_n]$ denote homogeneous coordinates on $\bbP^n_k,$ and suppose $x$ has homogeneous coordinates $[1, a_1, \ldots, a_n]$.  Let $f: \bbP_k^n \to \bbP_k^n$ be the automorphism $$f([X_0, X_1, \ldots, X_n] ) = [X_0, X_1 + a_1 X_0, \ldots, X_n + a_n X_0].$$ The diagram $$ \xymatrix{ \bbP_k^n/(\bbP_k^{n} - \{0\}) \ar[d]_f \ar[r]_{\cong}^{\puritytoPnbar} & \ar[d]^{1} \bbP_k^n/(\bbP_k^{n-1}) \\
\bbP_k^n/(\bbP_k^{n} - \{x\})\ar[r]_{\cong}^{\puritytoPnbar} & \bbP_k^n/(\bbP_k^{n-1})}$$ commutes by naturality of Purity \cite[Lemma 2.1]{Voevodsky_MCZ2} and the compatibility of the trivializations of $T_x \bbP^n_k$ and $T_0 \bbP^n_k$. The diagram $$\xymatrix { \bbP_k^n/(\bbP_k^{n-1}) \ar[d]_{\overline{f}} \ar[r] &\bbP_k^n/(\bbP_k^{n} - \{0\}) \ar[d]^f \\
\bbP_k^n/(\bbP_k^{n-1})  \ar[r] &\bbP_k^n/(\bbP_k^{n} - \{x\}) } $$ comparing collapse maps via the maps induced by $f$ commutes by definition. Since $$[X_0, X_1, \ldots, X_n] \times t \mapsto [X_0, X_1 + a_1 tX_0, \ldots, X_n + a_nt X_0] $$ defines a naive homotopy between $\overline{f}$ and the identity, it suffices to show the lemma when $x$ is the origin. This case follows from \cite[Lemma 2.2]{Voevodsky_MCZ2} and \cite[Proposition~2.17 proof of number 3, page~112]{morelvoevodsky1998}.

\end{proof}

In particular, for a $k$-rational point $x$, the collapse map $\bbP_k^n/(\bbP_k^{n-1}) \to \bbP_k^n/(\bbP_k^{n} - \{x\}) $ is an $\bbA^1$-homotopy equivalence.

\begin{df}\label{def:local_degree_f(x)_rational}
Let $f: \bbA^n_k \to \bbA^n_k$ be a morphism, and let $x$ be a closed point such that $x$ is isolated in its fiber $f^{-1}(f(x))$, and $f(x)$ is $k$-rational. The {\em local degree} (or local $\bbA^1$-Brouwer degree) $\deg^{\bbA^1}_x f$ of $f$ at $x$ is Morel's $\bbA^1$-degree homomorphism applied to a map $$ f_x: \bbP_k^n/\bbP_k^{n-1} \to \bbP_k^n/\bbP_k^{n-1} $$ in the homotopy category, where $f_x$ is defined to be the composition
$$ \bbP_k^n/\bbP_k^{n-1} \to \bbP_k^n/(\bbP_k^{n} - \{x\})  \stackrel{\cong}{\leftarrow}    U/(U - \{ x\}) \stackrel{f\vert_U}{\longrightarrow}  \bbP_k^n/(\bbP_k^n - \{f(x) \}) \stackrel{\cong}{\leftarrow}  \bbP_k^n/\bbP_k^{n-1} $$
\end{df}

When $x$ is a $k$-point, it is perhaps more natural to define the local degree in the following equivalent manner: the trivialization of the tangent space of $\bbA^n_k$ gives canonical $\bbA^1$-weak equivalences $U/(U - \{ x\}) \cong \Th(\mathcal{O}_k^n)$ and $ \bbA_k^n/(\bbA^n_k -\{ f(x)\})\cong \Th(\mathcal{O}_k^n)$ by Purity \cite[Theorem 2.23, page~115]{morelvoevodsky1998}. As above, we have a canonical $\bbA^1$-weak equivalence $\Th(\mathcal{O}_k^n) \cong \bbP_k^n/(\bbP_k^{n-1})$. The local degree of $f$ at $x$ is the degree of the map in the homotopy category $$f_x': \bbP_k^n/(\bbP_k^{n-1}) \cong U/(U -\{ x\})  \stackrel{f\vert_U}{\longrightarrow} \bbA_k^n/(\bbA^n_k -\{ f(x)\})\cong  \bbP_k^n/(\bbP_k^{n-1})$$ as we now show.

\begin{pr}\label{local_degree_alternate_def}
When $x$ is a $k$-point of $\bbA^n_k$ and $f: \bbA^n_k \to \bbA^n_k$ is a morphism
$$\deg^{\bbA^1}_x f = \deg f_x'$$
\end{pr}

\begin{proof}

Let $ c_{f(x)}^{-1}$ denote the inverse in the homotopy category of $c_{f(x)}$ as defined in Lemma \ref{crush=id}. The definitions produce the equality $ c_{f(x)}^{-1} f_x' c_x = f_x$, which implies the result by Lemma \ref{crush=id}.

\end{proof}

For $n = 1$, the following lemma is \cite[Lemma 5.5]{Hoyois_lef}, and Hoyois's proof generalizes to higher $n$ as follows.

\begin{lm}\label{collapse_is_unit}
Let $x$ be a closed point of $\bbA_k^n$. The collapse map $$ \bbP_k^n/(\bbP_k^{n-1}) \to \bbP_k^n/(\bbP_k^{n} - \{x\}) \cong \Th T_x \bbP_k^n \cong  \Th T_x \bbA_k^n \cong   \bbP_k^n/(\bbP_k^{n-1}) \wedge \Spec k(x)_+$$ is $ \bbP_k^n/(\bbP_k^{n-1}) \wedge (-)$ applied to the canonical map \begin{equation}\label{collapse_6_functors}\eta: 1_k \to p_* p^* 1_k \cong p_* 1_{k(x)}\end{equation} in $\Spt(k)$, where $p:\Spec k(x) \to \Spec k$ is the structure map, and the last equivalence is from \cite[3. Proposition~2.17, page~112]{morelvoevodsky1998}. 
\end{lm}

\begin{proof}
As above, consider the trivialization of $T_x \bbP^n_k$ coming from the isomorphism $T_x \bbP_k^n \cong T_x \bbA_k^n$ and the canonical trivialization of $T_x \bbA_k^n$. The closed immersion $x: \Spec k(x) \to \bbP_n^k$ and this trivialization determine a Euclidean embedding in the sense of Hoyois \cite[Definition 3.8]{Hoyois_lef}. This Euclidean embedding determines an isomorphism $   \bbP_k^n/(\bbP_k^{n} - \{x\}) \cong \bbP_k^n/(\bbP_k^{n-1}) \wedge p_! 1_{k(x)}$ in $\Spt(k)$ by \cite[3.9]{Hoyois_lef} and the identification $\Th \mathcal{O}_k^n \cong \bbP_k^n/(\bbP_k^{n-1}) $ of \cite[3. Proposition~2.17, page~112]{morelvoevodsky1998}. Since $p$ is finite \'etale, there is a canonical equivalence $p_! 1_{k(x)}\cong \Spec k(x)_+$, and these identifications agree with the isomorphism $\bbP_k^n/(\bbP_k^{n} - \{x\}) \cong  \bbP_k^n/(\bbP_k^{n-1}) \wedge \Spec k(x)_+ $ in the statement of the lemma. By \cite[Proposition 3.14]{Hoyois_lef}, it thus suffices to show that a certain composition \begin{equation}\label{comp_collapse_unit_pf}
\bbP_k^n/(\bbP_k^{n-1}) \wedge \Spec k(x)_+ \to \bbP_k^n/(\bbP_k^{n} - \{x\}) \wedge \Spec k(x)_+ \stackrel{h}{\to} \bbP_k^n/(\bbP_k^{n-1}) \wedge \Spec k(x)_+
\end{equation} of the collapse map $ \bbP_k^n/(\bbP_k^{n-1}) \to \bbP_k^n/(\bbP_k^{n} - \{x\})$ smash $\Spec k(x)_+$ with a map $h$ is the identity in $\Spt(k)$. 

To define $h$, introduce the following notation. Let $$x_{k(x)}: \Spec k(x) \otimes k(x) \hookrightarrow \bbP^n_{k(x)}$$ be the base change of $x$. Let $$\tilde{x} = x_{k(x)} \circ \Delta: \Spec k(x) \to \bbP_{k(x)}^n$$ be the composition of the diagonal with $x_{k(x)}$. Let $$r: \bbP_{k(x)}^n/(\bbP_{k(x)}^{n} - \{\tilde{x} \}) \stackrel{\cong}{\rightarrow} \bbP_{k(x)}^n/\bbP_{k(x)}^{n-1}$$ be as Lemma \ref{crush=id} with $k$ replaced by $k(x)$.  Using the identifications $\bbP_k^n/(\bbP_k^{n} - \{x\}) \wedge \Spec k(x)_+ \cong \bbP_{k(x)}^n/(\bbP_{k(x)}^{n} - x_{k(x)})$ and $\bbP_{k(x)}^n/(\bbP_{k(x)}^{n-1}) \cong \bbP_k^n/(\bbP_k^{n-1}) \wedge \Spec k(x)_+$, we can view $h$ as a map $$h:\bbP_{k(x)}^n/(\bbP_{k(x)}^{n} - x_{k(x)}) \to \bbP_{k(x)}^n/(\bbP_{k(x)}^{n-1}).$$ Then $h$ is the composition \begin{equation*}
\bbP_{k(x)}^n/(\bbP_{k(x)}^{n} - x_{k(x)}) \to \bbP_{k(x)}^n/(\bbP_{k(x)}^{n} - \{\tilde{x} \}) \stackrel{r}{\to} \bbP_{k(x)}^n/(\bbP_{k(x)}^{n-1}) 
\end{equation*} 

The composition \eqref{comp_collapse_unit_pf} is now identified with $p_{\sharp}$ applied to the composition in Lemma  \ref{crush=id} of the collapse map with $r$ for the rational point $\tilde{x}: \Spec k(x) \to \bbA_{k(x)}^n$, completing the proof by Lemma  \ref{crush=id}.
\end{proof}

The degree of an endomorphism of $\bbP_k^n/(\bbP_k^{n-1})$ is the sum of local degrees under the hypotheses of the following Proposition. 

\begin{pr} \label{pr_deg_is_sum_local_deg}
Let $f: \bbP_k^n \to \bbP_k^n$ be a finite map such that $f^{-1}(\bbA_k^n) = \bbA_k^n$, and let $\overline{f}$ denote the induced map $\bbP_k^n/(\bbP_k^{n-1}) \to \bbP_k^n/(\bbP_k^{n-1})$. Then for any $k$-point $y$ of $\bbA_k^n$, $$\deg^{\bbA^1}(\overline{f}) = \Sigma_{x \in f^{-1}(y)} \deg^{\bbA^1}_x(f).$$
\end{pr}

\begin{proof}
By Purity \cite[Theorem~2.23, page~115]{morelvoevodsky1998}, $\bbP_k^n/(\bbP_k^n - f^{-1} \{y \})$ is the Thom space of the normal bundle to $f^{-1} \{y \} \hookrightarrow \bbA_k^n$. The Thom space of a vector bundle on a disjoint union is the wedge sum of the Thom spaces of the vector bundle's restrictions to the connected components. It follows that the quotient maps $$ \bbP_k^n/(\bbP_k^n - f^{-1} \{y \}) \to \bbP_k^n/(\bbP_k^n - \{x \})$$ for $x$ in $f^{-1}(y)$ determine an $\bbA^1$-weak equivalence $$ \bbP_k^n/(\bbP_k^n - f^{-1} \{y \}) \to  \vee_{x \in f^{-1}(y)} \bbP_k^n/ (\bbP_k^n - \{x\})  .$$ There is a commutative diagram \begin{equation*}\xymatrix{ & \vee_{x \in f^{-1}(y)} \bbP_k^n/ (\bbP_k^n - \{x\}) \ar[rrd]&&  \\ \bbP_k^n/ (\bbP_k^n - \{x\})  \ar[ur]^{k_x} & \ar[l] \ar[u]^{\cong} \bbP_k^n/(\bbP_k^n - f^{-1} \{y \}) \ar[rr] && \bbP_k^n/(\bbP_k^n -\{y \})  \\ & \ar[ul]^{\cong} \ar[u] \bbP_k^n/ \bbP_k^{n-1} \ar[rr]_{\overline{f}} && \ar[u]^{\cong} \bbP_k^n/ \bbP_k^{n-1} }\end{equation*} Apply $[\bbP_k^n/ \bbP_k^{n-1}, -]_{\Spt(k)}$ to the above diagram, and let $\overline{f}_*$ be the induced map $$\overline{f}_*: [\bbP_k^n/ \bbP_k^{n-1}, \bbP_k^n/ \bbP_k^{n-1}]_{\Spt(k)}\to [\bbP_k^n/ \bbP_k^{n-1}, \bbP_k^n/ \bbP_k^{n-1}]_{\Spt(k)} .$$ Because the wedge and the product are stably isomorphic, $$[\bbP_k^n/ \bbP_k^{n-1}, \vee_{x \in f^{-1}(y)} \bbP_k^n/ (\bbP_k^n - \{x\})]_{\bbA^1}  \cong \oplus_{x \in f^{-1}(y)} [\bbP_k^n/ \bbP_k^{n-1},  \bbP_k^n/ (\bbP_k^n - \{x\})]_{\Spt(k)}$$ and on the right hand side, $k_x$ induces the inclusion of the summand indexed by $x$. The image of the identity map under $\overline{f}_*$ can be identified with $\deg  \overline{f}$. Using the outer composition in the commutative diagram, we see that the image of the identity map under $\overline{f}_*$ can also be identified with $ \Sigma_{x \in f^{-1}(y)} \deg_x f.$
\end{proof}

We now give a computation of the local degree at points where $f$ is \'etale. 

\begin{pr}\label{loc_degree_etale_point}
Let $f: \bbA_k^n \to \bbA_k^n$ be a morphism of schemes and $x$ be a closed point of $\bbA_k^n$ such that $f(x) = y$ is $k$-rational and $x$ is isolated in $f^{-1}(y)$. If $f$ is \'etale at $x$, then the local degree is computed by $$\deg^{\bbA^1}_x f =\Tr_{k(x)/k} \langle J (x) \rangle,$$ where $J(x)$ denotes the Jacobian determinant $J = \det \begin{pmatrix} \frac{\partial f_i}{\partial x_j} \end{pmatrix}$ evaluated at $x$, and $k(x)$ denotes the residue field of $x$.
\end{pr}

\begin{proof}

We work in $\Spt(k)$. Let $p: \Spec k(x) \to \Spec k$ denote the structure map. 

Since $f$ is \'etale at $x$, the induced map of tangent spaces $df(x): T_x \bbA_k^n \to f^* T_{f(x)} \bbA^n_k$ is a monomorphism. Thus $df(x)$ induces a map on Thom spectra, which factors as in the following commutative diagram (see Conventions \eqref{thfphi=comp}): \begin{equation}\label{Thdf_CD_in_loc_deg_et_pt_pf}\xymatrix{ \Th T_x \bbA_k^n \ar[rrrrr]^{\Th_f (df(x))} \ar[d]^{\cong}&&&&& \Th T_{f(x)} \bbA^n_k \ar[d]^{\cong} \\ \ar[u] \Th \mathcal{O}^n_{\Spec k(x)} \ar[rrr]^{\Th_{1_{\Spec k(x)}} \begin{pmatrix} \frac{\partial f_i}{\partial x_j} \end{pmatrix}} &&& \ar[rr]^{\Th_p 1_{p^*\mathcal{O}^n_{\Spec k} }} \Th \mathcal{O}^n_{\Spec k(x)}  && \ar[u] \Th \mathcal{O}^n_{\Spec k}} \end{equation}

The naturality of the Purity isomorphism  \cite[Lemma 2.1]{Voevodsky_MCZ2} gives the commutative diagram \begin{equation}\label{purity_in_loc_deg_etale_pt_pf}
\xymatrix{ U/(U-\{x\}) \ar[rr]^{f\vert_U} && \bbP^n_k / (\bbP^n_k - \{ f(x)\}) \\
\Th T_x U \ar[rr]^{\Th_f (df(x))} \ar[u]^{\cong}&& \Th T_{f(x)} \bbP^n_k \ar[u]^{\cong}.}
\end{equation}

The isomorphisms $T_x U \cong T_x \bbA_k^n$ and  $\Th T_{f(x)} \bbA^n_k \cong \Th T_{f(x)} \bbP^n_k $ allow us to stack Diagram \eqref{purity_in_loc_deg_etale_pt_pf} on top of Diagram \eqref{Thdf_CD_in_loc_deg_et_pt_pf}. We then expand the resulting diagram to express the map $f_x$ from Definition \ref{def:local_degree_f(x)_rational} in terms of $\begin{pmatrix} \frac{\partial f_i}{\partial x_j} \end{pmatrix}$.

\begin{equation}
\xymatrix{ \bbP^n_k/\bbP^{n-1}_k \ar@/^2pc/[rrrrr]^-{f_x} \ar[r] \ar[rrddd]_{ \bbP^n_k/\bbP^{n-1}_k\wedge \eta }&\bbP^n_k/\bbP^n_k - \{ x\} &\ar[l]^{\cong }U/(U-\{x\}) \ar[rr]^{f\vert_U} && \bbP^n_k / (\bbP^n_k - \{ f(x)\}) & \ar[l]^{\cong}  \bbP^n_k/\bbP^{n-1}_k \ar[dddl]^1\\
& & \Th T_x U \ar[rr]^{\Th_f (df(x))} \ar[u]^{\cong}&& \Th T_{f(x)} \bbP^n_k \ar[u]^{\cong} &\\
& &\Th \mathcal{O}^n_{\Spec k(x)} \ar[u]^{\cong} \ar[r] & \ar[r] \ldots& \Th \mathcal{O}^n_{\Spec k} \ar[u]^{\cong} &\\
& & \bbP^n_k/\bbP^{n-1}_k \wedge \Spec k(x)_+ \ar[u]^{\cong}&& \bbP^n_k/\bbP^{n-1}_k \ar[u]^{\cong} &.}
\end{equation}

We have applied Lemma \ref{collapse_is_unit} to identify the diagonal maps.  

We furthermore have an identification (see Conventions \eqref{Thf1_as_counit}) of $\Th_p 1_{p^*\mathcal{O}^n_{\Spec k} }$ with the composition \begin{equation*}
p_! \Sigma^{p^* \mathcal{O}^n_{\Spec k}} 1_{k(x)} \to \Sigma^{\mathcal{O}^n_{\Spec k}} p_! 1_{k(x)} \stackrel{\epsilon}{\rightarrow} \Sigma^{\mathcal{O}^n_{\Spec k}} 1_k .
\end{equation*} Since $p$ is \'etale, there is a canonical identification $p_! \cong p_{\sharp}$.

We may therefore identify $f_x$ with the composition 
$$\bbP^n_k/ \bbP^{n-1}_k \stackrel{\bbP^n_k/ \bbP^{n-1}_k \wedge \eta}{\rightarrow} \bbP^n_{k(x)}/ \bbP^{n-1}_{k(x)} \cong \Th \mathcal{O}^n_{\Spec k(x)} \stackrel{\begin{pmatrix} \frac{\partial f_i}{\partial x_j} \end{pmatrix}}{\to}  \Th \mathcal{O}^n_{\Spec k(x)} \cong  \bbP^n_{k(x)}/ \bbP^{n-1}_{k(x)}  \stackrel{\bbP^n_k/ \bbP^{n-1}_k \wedge \epsilon}{\rightarrow} \bbP^n_k/ \bbP^{n-1}_k$$

By \cite[Lemma 5.3]{Hoyois_lef}, we therefore have that $f_x $ is the trace of the endomorphism $\begin{pmatrix} \frac{\partial f_i}{\partial x_j} \end{pmatrix}$ of $\Th \mathcal{O}^n_{\Spec k(x)}$ in $\Spt(k)$. It follows that $\deg f_x = \Tr_{k(x)/k} \langle J \rangle$ by \cite[Theorem 1.9]{Hoyois_lef}.

\end{proof}

\section{Some finite determinacy results } \label{Section: FinitelyDetermined}
Here we prove a finite determinacy result and then use that result to prove a result, Proposition~\ref{Prop: ReduceToRegSeq_enhanced}, that allows us to reduce the proof of the Main Theorem to a case where $f$ is \'{e}tale at $0$.  In this section we fix a polynomial function $f \colon \bbA^{n}_{k} \to \bbA^{n}_{k}$ that has an isolated zero at the origin and write $f_1, \dots, f_n \in P$ for the component functions. 

The finite determinacy result is as follows.

\begin{df}
	Let $f, g \colon \bbA^{n}_{k} \to \bbA^{n}_{k}$ be polynomial functions.  Then we say that $f$ and $g$ are \textbf{equivalent at the origin} if both functions have isolated zeros at the origin and we have
	\begin{enumerate}
		\item  $Q_{0}(f)=Q_{0}(g)$ and $E_{0}(f) =E_{0}(g)$;
		\item $\deg^{\bbA^1}_{0}(f) = \deg^{\bbA^1}_{0}(g)$.
	\end{enumerate}
	
	We say that a polynomial function $f \colon \bbA^{n}_{k} \to \bbA^{n}_{k}$ with an isolated zero at the origin is \textbf{$b$-determined} if every polynomial function $g$ with the property that $f_i - g_i \in \mathfrak{m}_{0}^{b+1}$ for $i=1, \dots, n$ is equivalent to $f$.  We say that $f$ is \textbf{finitely determined} if it is $b$-determined for some $b \in \mathbb{N}$.
\end{df}

\begin{lm} \label{Lemma: FiniteDeterminacy}
	A polynomial function $f \colon \bbA^{n}_{k} \to \bbA^{n}_{k}$ with an isolated zero at the origin is finitely determined.
\end{lm}
\begin{proof}
	Since $Q_{0}(f)$ is a finite length quotient of $P_{\mathfrak{m}_{0}}$, its defining ideal must contain a power of the maximal ideal, say $\mathfrak{m}_{0}^{b} \subset (f_1, \dots f_n)$.  We will prove that any $g$ satisfying $f_i - g_i \in \mathfrak{m}_{0}^{b+1}$ satisfies the desired conditions.  To begin, we show the ideals $(f_1, \dots, f_n)$ and $(g_1, \dots, g_n)$ are equal.  By the choice of $b$, we have $(g_1, \dots, g_n) \subset (f_1, \dots, f_n)$, and to see the reverse inclusion, we argue as follows.  The elements $g_1, \dots, g_n$ generate $(g_1, \dots, g_n) + \mathfrak{m}_{0}^{b}$ modulo $\mathfrak{m}_{0}^{b+1}$ (since, modulo $\mathfrak{m}_{0}^{b+1}$, the $g_i$'s equal the $f_i$'s), so by Nakayama's lemma, the $g_i$'s generate $(g_1, \dots, g_n) + \mathfrak{m}_{0}^{b}$. In particular, $(g_1, \dots, g_n) \supset \mathfrak{m}_{0}^{b}$, and we conclude $(g_1, \dots, g_n) \supset (f_1, \dots, f_n)$.
	
	We immediately deduce $Q_{0}(f)= Q_{0}(g)$.  To see that $E_{0}(f) = E_{0}(g)$, observe that if we write $g_i = f_i + \sum b_{i,j} x_j$ with $b_{i, j} \in \mathfrak{m}_{0}^{b}$ and $f_i= \sum a_{i, j} x_j$, then $g_{i} = \sum (a_{i, j} + b_{i, j}) x_{j}$ and
	\[
		E_{0}(g) = \det( a_{i, j} + b_{i, j}).
	\]
	Since $a_{i, j}$ equals $a_{i, j} + b_{i, j}$ modulo $\mathfrak{m}^{b}_{0}$, these elements are equal modulo $(f_1, \dots, f_n) = (g_1, \dots, g_n)$.  Taking determinants, we conclude $E_{0}(f) = E_{0}(g)$.  This proves (1).
	
	We prove (2) by exhibiting an explicit naive $\bbA^1$-homotopy  between the maps $f'_{0}$ and $g'_{0}$ on Thom spaces. Write  
	\[
		g_i = \sum n_{i, j} f_{j} \text{ in $P_{\mathfrak{m}_{0}}$.}
	\]
	By definition, $f_i = g_i \text{ modulo $\mathfrak{m}_{0}^{b+1}$}$, hence modulo $\mathfrak{m}_{0} \cdot (f_1, \dots, f_n)$.  Moreover, $f_1, \dots, f_n$ is a basis for the $k$-vector space $(f_1, \dots, f_n)/ \mathfrak{m}_{0} \cdot (f_1, \dots, f_n)$, so the matrix $(n_{i, j})$ must reduce to the identity matrix modulo $\mathfrak{m}_{0}$, allowing us to write $(n_{i, j}) = \operatorname{id}_{n} +(m_{i, j})$ with $m_{i, j} \in \mathfrak{m}_{0}$.  Let $V \subset \bbA^{n}_{k}$ be a Zariski neighborhood of the origin such that the entries of the matrix $  \operatorname{id}_{n} + (m_{i, j})$ are restrictions of elements of $H^{0}(V, \calO)$ that we denote by the same symbols.  Now consider the matrix
	\[
		M(x,t) := \operatorname{id}_{n} +  ( t \cdot m_{i, j}(x))
	\]
	and the map 
	\begin{gather*}
		H \colon  V \times_{k} \bbA_{k}^{1} \to \bbA^{n}_{k}, \\
		(x,t) \mapsto M(x, t) \cdot f(x).
	\end{gather*}
	The  preimage $H^{-1}(0)$ contains $\{ 0 \} \times_{k} \bbA^{1}_{k}$ as a connected component.  Indeed, to see this is a connected component, it is enough to show that the subset is open.  To show this, observe that the complement set is 
	\begin{displaymath}
		\left( \{ \det(M)=0 \} \cup (f^{-1}(0) - 0) \times_{k} \bbA^{1}_{k} \right) \cap H^{-1}(0).
	\end{displaymath}
	The subset $\{ \det(M) =0 \}$ is closed in $V \times_{k} \bbA^{1}_{k}$ as $\det(M)$ is a regular function, and $f^{-1}(0) - \{ 0 \} \subset \bbA^{n}_{k}$ is closed as $0 \in f^{-1}(0)$ is a connected component by hypothesis.

	The map $H$ induces a map on quotient spaces
	\begin{equation} \label{Eqn: ConstructHomotopy}
		H \colon \frac{V \times \bbA^{1}_{k}}{V \times \bbA^{1}_{k} - H^{-1}(0)} 	\to \frac{\bbA^{n}_{k}}{\bbA^{n}_{k} - 0}.
	\end{equation}
	The quotient $\frac{V \times \bbA^{1}_{k}}{V \times \bbA^{1}_{k} - 0 \times \bbA^{1}_{k}}$ naturally includes into $\frac{V \times \bbA^{1}_{k}}{V \times \bbA^{1}_{k} - H^{-1}(0)}$ because, as $0 \times \bbA^{1}_{k}$ is a connected component of $H^{-1}(0)$,  $\frac{V \times \bbA^{1}_{k}}{V \times \bbA^{1}_{k} - H^{-1}(0)}$ is canonically identified with the  wedge sum of $\frac{V \times \bbA^{1}_{k}}{ V \times \bbA^{1}_{k} - 0 \times \bbA^{1}_{k}}$ and $\frac{V \times \bbA^{1}_{k}}{ V \times \bbA^{1}_{k} - W}$, where $W$ is the complement of $0 \times \bbA^{1}_{k}$ in $H^{-1}(0)$.  Consider now the composition of the inclusion with $H$:
	\begin{equation} \label{Eqn: ConstructHomotopy2}
		H \colon \frac{V \times \bbA^{1}_{k}}{V \times \bbA^{1}_{k} - 0 \times_{k} \bbA^{1}_{k}} 	\to \frac{\bbA^{n}_{k}}{\bbA^{n}_{k} - 0}.
	\end{equation}
	
	The spaces appearing in this last equation  are identified with the Thom spaces of normal bundles by the purity theorem, and these Thom spaces, in turn, are isomorphic to  smash  products with $\Th(\calO^{\oplus n}_{\Spec{k}})$ because the relevant normal bundles are trivial:
	\begin{align*}
		\frac{\bbA^{n}_{k}}{\bbA^{n}_{k} - 0}						 =& 	\Th(\{ 0 \} ) \\
														=& 	\{ 0 \}_{+} \wedge \Th(\calO_{\Spec(k)}^{\oplus n}), \\
		\frac{V \times \bbA^{1}_{k}}{V \times \bbA^{1}_{k} - 0 \times \bbA^{1}_{k}} 	=& \Th(0 \times \bbA^{1}_{k}) \\
																=& (0 \times \bbA^{1}_{k})_{+} \wedge \Th(\calO_{\Spec(k)}^{\oplus  n}).
	\end{align*}
	These identifications identify \eqref{Eqn: ConstructHomotopy2} with a naive $\bbA^1$-homotopy 
	\[	
		(0 \times \bbA^{1}_{k})_{+} \wedge \Th(\calO_{\Spec(k)}^{\oplus n}) \to \{ 0 \}_{+} \wedge \Th(\calO_{\Spec(k)}^{\oplus n})
	\]
	from $f'_{0}$ to $g'_0$.
\end{proof}

\begin{rmk}
	When $f$ is a polynomial function in $1$ variable, we can exhibit an explicit $b$.  Indeed, $f$ is $b$-determined provided $f$ contains a nonzero monomial of degree $b$.  To see this, take $b$ to be the least such integer and write $f = u \cdot x^{b}$ for $u \in k[x]$ a unit in $(P_{x})_{\mathfrak{m}_{0}}$.  The ideal $(x^b)$  lies in $(f)$, so the proof of Lemma~\ref{Lemma: FiniteDeterminacy} shows that $f$ is $b$-determined. We make use of this fact in the companion paper \cite{wickelgren16}.   
\end{rmk}

In the proof of the Main Theorem, we use Proposition~\ref{Prop: ReduceToRegSeq_enhanced} to reduce the proof to the special case where the following assumption holds:
\begin{assumption} \label{Assumption}
	The polynomial function $f $ is the restriction of a morphism $F \colon \bbP^{n}_{k} \to \bbP^{n}_{k}$ such that
	\begin{enumerate}
		\item $F$ is finite, flat, and with induced field extension $\operatorname{Frac} F_{*}\calO_{\bbP^{n}_{k}} \supset  \operatorname{Frac} \calO_{\bbP^{n}_{k}}$ of degree coprime to $\operatorname{char}(k)=p$;
		\item $F$ is \'etale at every point of $F^{-1}(0) - \{ 0 \}$;
		\item $F$ satisfies $F^{-1}(\bbA^n_{k}) \subset \bbA^{n}_{k}$.
	\end{enumerate}
\end{assumption}

To reduce to the special case, we need to prove that, after possibly passing from $k$ to an odd degree field extension, every $f$ is equivalent to a polynomial function satisfying Assumption~\ref{Assumption}, and we conclude this section with a proof of this fact.  The proof we give below is a modification of \cite[Theorem 4.1]{Becker}, a theorem about  real polynomial functions due to Becker--Cardinal--Roy--Szafraniec.  

We will show that if $f$ is a given polynomial function with an isolated zero at the origin, then for a general $h \in P$ that is homogeneous and of degree sufficiently large and coprime to $p$, the sum $f+h$ satisfies Assumption~\ref{Assumption}.  This result is Proposition~\ref{Prop: ReduceToRegSeq_enhanced}.  We prove that result as a result about the affine space $H^{d}_{k}$  parameterizing polynomial maps $h \colon \bbA^{n}_{k} \to \bbA^{n}_{k}$ given by $n$-tuples of homogeneous degree $d$ polynomials.  We show that the locus of $h$'s such that $f+h$ fails to  satisfy Assumption~\ref{Assumption} is not equal to $H^{d}_{k}$  by using the following three lemmas.

\begin{lm}\label{modify_to_give_endo_Pn}
The subset of $H^{d}_{k}$ corresponding to $h$'s such that $h^{-1}(0)=0$ is a nonempty Zariski open subset. 
\end{lm}
\begin{proof}
	The subset in question is nonempty since it contains e.g.~$(x_1^d, \dots, x_n^d)$.  To see that it is open, consider the image $I \subset \bbP( H^{d}_{k})$ of the incidence variety 
	\[
		\{ ([h_1, \dots, h_n], [x_1, \dots, x_{n}]) \colon h_{1}(x_1, \dots, x_n) = \dots = h_{n}(x_1, \dots, x_n)=0 \} \subset \bbP(H^{d}_{k}) \times \bbP_{k}^{n-1}.
	\]
	under  the projection $\bbP(H^{d}_{k}) \times \bbP_{k}^{n-1} \to \bbP (H^{d}_{k})$.  The subset $I$ is closed by the properness of the projection morphism.  The preimage of $I$ under the natural morphism $H^{d}_{k}-0 \to \bbP(H^{d}_{k})$ is closed in $H^{d}_{k}-0$, hence the union of the preimage of $I$ and $0$ is closed in $H^{d}_{k}$.  The complement of this closed subset is the subset of $h$'s such that $h^{-1}(0)=0$.
\end{proof}

The following lemma is used in Lemma~\ref{modify_to_have_regular_values} to bound a dimension.

\begin{lm} \label{Lemma: BoundFibers}
	Let $f \colon \bbA^{n}_{k} \to \bbA^{n}_{k}$ be a nonzero polynomial function satisfying $f(0)=0$ and  $a=(a_1, \dots, a_n) \in \bbA^{n}_{k}(k)$  a $k$-point that is not the origin $0$.  If $\sum \frac{\partial f_1}{\partial x_i}(a) \cdot a_i \ne d \cdot f_{1}(a)$ for $d$ an integer that is nonzero in $k$, then the subset of $H_{k}^{d}$ consisting of $h$'s satisfying
	\begin{align}
		f_{i}(a) + h_{i}(a) =& 0 \text{ for $i=1, \dots, n$,} \label{Eqn: RegSeqOne} \\
		\det \begin{pmatrix} \frac{\partial(f_i + h_i) }{\partial x_j}(a)  \end{pmatrix} =& 0 \label{Eqn: RegSeqTwo}
	\end{align}
	is a Zariski closed subset of codimension $n+1$.
\end{lm}
\begin{proof}
	It is enough, by the Krull principal ideal theorem, to show that \eqref{Eqn: RegSeqOne} and \eqref{Eqn: RegSeqTwo}, considered as regular functions on $H_{k}^{n}$, form a regular sequence. Writing 
	\begin{gather*}
		h_1 = \sum c_{\underline{i}}(1) x_1^{i_1} x_{2}^{i_2} \dots x_{n}^{i_n}, \\
		h_2 = \sum c_{\underline{i}}(2) x_1^{i_1} x_{2}^{i_2} \dots x_{n}^{i_n}, \\
		\dots \\
		h_n = \sum c_{\underline{i}}(n) x_1^{i_1} x_{2}^{i_2} \dots x_{n}^{i_n}, \\
	\end{gather*}
	the coefficients $\{ c_{\underline{i}}(1), \dots, c_{\underline{i}}(n) \}$ are coordinates on the affine space $H_{k}^{d}$.  As polynomials in these coefficients, the elements $f_{1}(a)+ h_{1}(a), \dots, f_{n}(a) + h_{n}(a)$ from \eqref{Eqn: RegSeqOne} are affine linear equations, and distinct linear equations involve disjoint sets of variables, so we conclude that the first set of elements form a regular sequence with quotient equal to a polynomial ring.  In particular, the quotient is a domain, so to prove the lemma, it is enough to show that \eqref{Eqn: RegSeqTwo} has nonzero image in the quotient ring.
	
	To show this, first make a linear change of variables so  that $a = (1, 0,  \dots, 0)$.  Setting $\underline{i}_1 = (d, 0, \dots, 0)$, the elements \eqref{Eqn: RegSeqOne} take the form 
	\begin{displaymath} \label{Eqn: FirstElements}
		f_1(a)+c_{\underline{i}_{1}}(1), \dots, f_{n}(a)+c_{\underline{i}_{1}}(n),
	\end{displaymath}
	and if we let
	\begin{gather*}
		\underline{i}_{2} = (d-1, 1, 0, \dots, 0, 0), \\
		\underline{i}_{3} = (d-1, 0, 1, \dots, 0, 0), \\
		\dots \\
		\underline{i}_{n} = (d-1, 0, 0, \dots, 0, 1), \\
	\end{gather*}
	then \eqref{Eqn: RegSeqTwo} can be rewritten as 
	\[
		\det\begin{pmatrix}
			d \cdot c_{\underline{i}_{1}}(1) + \frac{\partial f_1}{\partial x_1}(a)	&	 c_{\underline{i}_{2}}(1) + \frac{\partial f_1}{\partial x_2}(a)	&	\dots		&	 c_{\underline{i}_{n}}(1) + \frac{\partial f_1}{\partial x_n}(a)	\\
			c_{\underline{i}_{1}}(2) + \frac{\partial f_2}{\partial x_1}(a)	&	d \cdot c_{\underline{i}_{2}}(2) + \frac{\partial f_2}{\partial x_2}(a)	&	\dots		&	c_{\underline{i}_{n}}(2) + \frac{\partial f_2}{\partial x_n}(a)	\\
			\vdots												&	\vdots												&	\ddots	&	\vdots		\\
			c_{\underline{i}_{1}}(n) + \frac{\partial f_n}{\partial x_1}(a)	&	c_{\underline{i}_{2}}(n) + \frac{\partial f_n}{\partial x_2}(a)	&	\dots		&	d \cdot c_{\underline{i}_{n}}(n) + \frac{\partial f_n}{\partial x_n}(a)
		\end{pmatrix}.
	\]
	This determinant is essentially the determinant of the general $n$-by-$n$ matrix $\det(x_{\alpha, \beta})$, as we now explain.
	
	Identify $\calO_{H_{k}^{d}}$ with $k[ x_{\alpha, \beta}]$  by setting, for $\alpha, \beta = 1, \dots, d$, the variable  $x_{\alpha, \beta}$ equal to the $(\beta, \alpha)$-th entry in the above matrix  (and, say, arbitrarily matching the remaining variables $c_{\underline{i}_{\beta}}(\alpha)$, $i_{\beta} \ne \underline{i}_{1}, \dots, \underline{i}_{n}$, in $\calO_{H_{k}^{d}}$ with the remaining variables in $k[x_{\alpha, \beta}]$).  This identification identifies the determinant under consideration with the determinant $\det(x_{\alpha, \beta})$ of the general $n$-by-$n$ matrix and identifies  the elements \eqref{Eqn: RegSeqOne} with linear polynomials, say $A_{1} x_{1, 1} + B_{1}, A_{2} x_{2, 1} + B_{2}, \dots, A_{n} x_{n, 1} + B_{n}$.  

Now consider $\det(x_{\alpha, \beta})$ as a function $\det(v_1, \dots, v_n)$ of  the column vectors.  Under the identification of  \eqref{Eqn: RegSeqTwo} with $\det(x_{\alpha, \beta})$, the image of \eqref{Eqn: RegSeqTwo} in the quotient ring is identified with $\det(\overline{v}_1, v_{2}, \dots, v_n)$ for  $\overline{v}_{1} = (-B_{1}/A_{1}, -B_{2}/A_{2}, \dots, -B_{n}/A_{n})$.  By the hypothesis $\sum \frac{\partial f_1}{\partial x_i}(a) \cdot a_i \ne d \cdot f_{1}(a)$, so   $-B_{1}/A_{1} \ne 0$ and $\overline{v}_{1}$ is not the zero vector.  We conclude that  $\det(\overline{v}_1,v_{2},  \dots, v_n)$ is nonzero because e.g.~we can extend $\overline{v}_{1}$ to a basis  $\overline{v}_{1}, \dots, \overline{v}_{n}$ and then $\det(\overline{v}_{1}, \overline{v}_{2} \dots, \overline{v}_{n}) \ne 0 $ by the fundamental property of the determinant.
\end{proof}

\begin{lm}\label{modify_to_have_regular_values}
Let $f \colon \bbA^{n}_{k} \to \bbA^{n}_{k}$ be a nonzero polynomial map satisfying $f(0)=0$ and $d$ an integer greater than the degrees of $f_1, \dots, f_n$ and coprime to $p$.  Then the subset $S \subset H^{d}_{k}$ of $h$'s such that  $f+h$ is \'{e}tale at every point of $(f+h)^{-1}(0) - 0$ contains a nonempty Zariski open subset.
\end{lm}
\begin{proof}
We prove the lemma by proving that the Zariski  closure of the complement of $S$ in $H^{d}_{k}$ has dimension strictly smaller than $\dim H^{d}_{k}$, hence the complement of $S$ cannot be Zariski dense.  

Consider 
\[
	\Delta  = \left\{ (h, a) \colon (f+h)(a)=0, \det \begin{pmatrix} \frac{\partial(f_i + h_i) }{\partial x_j}(a)  \end{pmatrix}  = 0, a \ne 0 \right\} \subset H_{k}^{d} \times (\bbA^{n}_{k}-0).
\]
The  complement of $S$ is the image $\pi_{1}(\Delta)$ of $\Delta$ under the first projection $\pi_1 \colon H_{k}^{d} \times \bbA^{n}_{k} \to H_{k}^{d}$.  We bound dimensions by analyzing the second projection $\pi_2 \colon H_{k}^{d} \times \bbA^{n}_{k} \to \bbA^{n}_{k}$.  

To bound the dimension, we argue as follows.  Some $f_i$ is nonzero since $f$ is nonzero, and without loss of generality, we can assume $f_1 \ne 0$.  Because $d$ is coprime to $p$ and strictly larger than the degree of $f_1$, the polynomial $\sum \frac{\partial f_1}{\partial x_i}(x) \cdot x_i - d \cdot f_{1}(x)$ is nonzero (by Euler's identity).  We conclude that
\[
	B := \{  a \in \bbA^{n}_{k} \colon  \sum \frac{\partial f_1}{\partial x_i}(a) \cdot a_i = d \cdot f_{1}(a) \}.
\]
has codimension $1$ in $\bbA^{n}_{k}$.  We separately bound $\Delta \cap \pi_{2}^{-1}(B)$ and $\Delta \cap \pi_2^{-1}  (\bbA^n_k - B)$.

The fibers of $\pi_{2} \colon \Delta - \pi_{2}^{-1}(B) \to \bbA^{n}_{k}-B$  have codimension $n+1$ by Lemma~\ref{modify_to_give_endo_Pn}, so by \cite[Proposition~5.5.2]{egaIV_2}, we have
\begin{align*}
	\dim \Delta \cap \pi_2^{-1}  (\bbA^n_k - B) 		\leq&		\dim  (\bbA^n_k - B) + \dim \Delta_a \\
										=&	 	n + \dim H_k^d - (n+1)\\
										<&  \dim H_k^d.
\end{align*}

By similar reasoning
\begin{align*}
	\dim \pi_2^{-1}(B) \cap \Delta 		\leq&  \dim B + \dim \Delta_{a} \\
								\leq&  n-1 +   \dim H_k^d - n \\
								<& \dim H_k^d.
\end{align*}

We conclude that $\pi_{1} \colon \Delta \to H_{k}^{d}$ cannot be dominant for dimensional reasons \cite[Theorem~4.1.2]{egaIV_2}.  The complement of the closure of $\pi_{1}(\Delta)$ can thus be taken as the  desired Zariski open subset.
\end{proof}

\begin{pr} \label{Prop: ReduceToRegSeq_enhanced}
	Let $f \colon \bbA^{n}_{k} \to \bbA^{n}_{k}$ be a nonzero polynomial function satisfying $f(0)=0$. Then there exists an odd degree extension $L/k$ such that $f \otimes_{k} L$ is equivalent to a function satisfying Assumption~\ref{Assumption}.  If $k$ is infinite, we can take $L=k$.
\end{pr}

\begin{proof}
The function $f$ is finitely determined by Lemma~\ref{Lemma: FiniteDeterminacy}, so say it is $b$-determined for $b \in \bbZ$.  Choose $d$ to be an integer coprime to $p$ and larger than both $b$ and the degrees of the $f_i$'s.   We claim that there exists an odd degree field extension $L/k$ and a  degree $d$ homogeneous polynomial function  $h \in H_{k}^{d}(L)$  such that $h^{-1}(0) = \{ 0 \}$ and $g := (f \otimes_{k} L)+h$ is \'{e}tale at every point of $g^{-1}(0)-0$.  To verify the claim, observe that  Lemmas \ref{modify_to_give_endo_Pn} and \ref{modify_to_have_regular_values} imply that the subset of all such $h$'s contains a nonempty Zariski open subset $U \subset H_{k}^{d}$.  If $k$ is an infinite field, $U(k)$ must be non-empty, so we take $L=k$.  Otherwise, $k$ is a finite field, say $k=\bbF_{q}$. We then have that $U(\bbF_{q^n})$ is nonempty for $n$ a sufficiently large odd number as $U(\overline{\bbF}_{q}) = \cup U(\bbF_{q^n})$.  The function $f \otimes_{k} L $ is also $b$-determined, and we complete the proof by showing that  $g := (f\otimes_{k}L) +h$  satisfies Assumption~\ref{Assumption}. 

To ease notation, we only give the proof in the case $L=k$ (the general case involves only notational changes).  Define degree $d$ homogeneous polynomials 
\begin{align*}
		G_0 :=& X_0^d,   \\
		G_1 :=& X_0^{d}\cdot f_1(X_1/X_0, \dots, X_n/X_0)+h_1(X_1, \dots, X_n),  \\
		\dots  \\
		G_n :=& X_0^{d}\cdot f_n(X_1/X_0, \dots, X_n/X_0)+h_n(X_1, \dots, X_n).  
\end{align*} 
The only solution $(X_0, \dots, X_n)$ to $G_0 = G_1 = \dots = G_n=0$ over $\overline{k}$ is $(0, \dots, 0)$ because from the first equation we deduce $X_{0}=0$ and then from the remaining equations we deduce that $(X_1, \dots, X_n)$ lies in $h^{-1}(0) = \{ 0 \}$.  It follows from general formalism that  $G = [G_0, G_1, \dots, G_n]$ defines a morphism $G \colon \bbP^{n}_{k} \to \bbP^{n}_{k}$ such that $G^{*}(\calO(1)) = \calO(d)$.  Moreover, by construction $G$ is an extension of $g$ that satisfies $G^{-1}(\bbA^{n}_{k}) \subset \bbA^{n}_{k}$.

To complete the proof, we need to show that $G$ is finite, flat, and induces a field extension $\operatorname{Frac} \calO_{\bbP^{n}_{k}} \subset \operatorname{Frac} G_{*}\calO_{\bbP^{n}_{k}}$ of degree coprime to $p$. To see that $G$ is finite, observe that the pullback $G^{-1}(\mathcal{H})$ of a hyperplane $\mathcal{H} \subset \bbP^{n}_{k}$ has positive degree on every curve (since the associated line bundle is $G^{*} \calO(1) = \calO(d)$, an ample line bundle).  We conclude that a fiber of $G$ cannot contain a curve since $G^{-1}(\mathcal{H})$ can be chosen to be disjoint from a given fiber.  In other words, $G$ has finite fibers. Being a morphism of projective schemes, $G$ is also proper and hence finite by Zariski's main theorem.  This implies that $G$ is flat since every finite morphism $\bbP^{n}_{k} \to \bbP^{n}_{k}$ is flat by  \cite[Corollary to Theorem~23.1]{Matsumura_CRT}.  %Section~23, ,  page~179

Finally, we complete the proof by noting that the degree of $\operatorname{Frac} \calO_{\bbP^{n}_{k}} \subset \operatorname{Frac} G_{*}\calO_{\bbP^{n}_{k}}$ equals  $d^n$, the top intersection number $G^{*}(\mathcal{H}^{n}) = \int c_{1}(\calO(d))^{n}$.  (To deduce the equality, observe that $\calH^{n}$ is the class of a $k$-point $\overline{y} \in \bbP^{n}_{k}(k)$, so the top intersection number is the $k$-rank of $\calO_{G^{-1}(\overline{y})}$, the stalk of $G_{*}\calO_{\bbP^{n}_{k}}$ at $\overline{y}$.  The rank of that stalk is equal to the rank of any other stalk of $G_{*}\calO_{\bbP^{n}_{k}}$  since  $G_{*}\calO_{\bbP^{n}_{k}}$, being finite and flat, is locally free.  In particular, that rank equals the rank of the generic fiber, which is  the degree of $\operatorname{Frac} G_{*}\calO_{\bbP^{n}_{k}} \supset \operatorname{Frac} \calO_{\bbP^{n}_{k}}$.)
\end{proof}

\section{The family of symmetric bilinear forms}  \label{Section: Family}
In this section we construct, for a given finite polynomial map $f \colon \bbA^{n}_{k} \to \bbA^{n}_{k}$, a family of symmetric bilinear forms over $\bbA^{n}_{k}$ such that the fiber over the origin contains  a summand that represents the ELK class $w_{0}(f)$.  This family has the property that the stable isomorphism class of the fiber over $\overline{y} \in \bbA^{n}_{k}(k)$ is independent of $\overline{y}$, and we use this property in Section~\ref{Section: PrOfMainThm} to compute $w_{0}(f)$ in terms of a regular value.  Finally, we compute the stable isomorphism class of the family over an \'{e}tale fiber.

Throughout this section $f$ denotes a finite polynomial map, except in Remark~\ref{Remark: WhyFiniteAssumption} where we explain what happens if the finiteness condition is weakened to quasi-finiteness.

The basic definition is the following.
\begin{df} \label{Def: FamilyOfForms}
	Define the \textbf{family of algebras} $\widetilde{Q} = \widetilde{Q}(f)$ associated to $f$ by 
	\begin{align*}
		\widetilde{Q} :=& f_{*}\calO_{\bbA^{n}_{k}}.
	\end{align*}
\end{df}
Concretely $\widetilde{Q}$ is the ring $P_{x}$ considered as a $P_{y}$-algebra by the homomorphism $y_1 \mapsto f_{1}(x), \dots, y_n \mapsto f_{n}(x)$ or equivalently the algebra $P_y[x_1, \dots, x_n]/(f_1(x)-y_1, \dots, f_{n}(x)-y_n)$.    Given $(\overline{y}_{1}, \dots, \overline{y}_{n}) = \overline{y} \in \bbA^{n}_{k}(L)$ for some field extension $L/k$, the fiber $\widetilde{Q} \otimes k(\overline{y})$ is the $L$-algebra $L[x_1, \dots, x_n]/(f_{1}(x)-\overline{y}_{1}, \dots, f_{n}(x)-\overline{y}_{n})$.  This algebra decomposes as
\[
	\widetilde{Q} \otimes k(\overline{y}) = Q_{x_1}(f) \times \dots \times Q_{x_m}(f),
\]
where $Q_{x}(f)$ is as in Definition~\ref{Section: LocalForm} and the product runs over all closed points $x \in f^{-1}(\overline{y})$.

The algebra $\widetilde{Q}$ has  desirable properties because we have assumed that $f$ is finite.
\begin{lm} \label{Lemma: RelativeCompleteIntersection}
	The elements $f_1(x)-y_1, \dots, f_{n}(x) - y_n \in P_{y}[x_1, \dots, x_n]$ form a regular sequence, and $\widetilde{Q}$ is $P_{y}$-flat.
\end{lm}
\begin{proof}
	It is enough to show that, for any maximal ideal $\mathfrak{m} \subset P_{y}$, the images of  $y_1-f_{1}(x),\dots, y_n - f_{n}(x)$ in $(P_{y}/\mathfrak{m})[x_1, \dots, x_n]$ form a regular sequence by \cite[first Corollary, page~177]{Matsumura_CRT}.  The quotient of $P_{y}[x_1, \dots, x_n]$ by the sequence is the structure ring of $f^{-1}(\mathfrak{m})$, which is $0$-dimensional by hypothesis.  In particular, the images of $y_1-f_1(x), \dots, y_n-f_{n}(x)$ generate a height $n$ ideal, and hence they form a regular sequence by \cite[Theorem~17.4(i)]{Matsumura_CRT}.
\end{proof}

As we mentioned in Section~\ref{Section: LocalForm}, Scheja--Storch have constructed a distinguished symmetric bilinear form $\beta_0$ on $Q_{0}(f)$ that represents $w_{0}(f)$.  In fact, they construct a family $\widetilde{\beta}$ of symmetric bilinear forms on $\widetilde{Q}$ such that the fiber over $0$ contains a summand that represents $w_{0}(f)$.  The family is defined as follows. 

\begin{df}	\label{Definition: SSForm}
	Let  $\widetilde{\eta} \colon \widetilde{Q} \to P_{y}$  be the generalized trace function, the $P_y$-linear function defined on \cite[page~182]{scheja}. Let  $\widetilde{\beta}$ be the symmetric bilinear form $\widetilde{\beta} \colon \widetilde{Q} \to P_{y}$ defined by $\widetilde{\beta}(\widetilde{a}_{1}, \widetilde{a}_{2}) = \widetilde{\eta}(\widetilde{a}_{1} \cdot \widetilde{a}_{2})$.  Given $\overline{y} \in \bbA^{n}_{k}(L)$ and $x \in f^{-1}(\overline{y})$, we write $\eta_{x}$ and $\beta_{x}$ for  the respective restrictions to $Q_{x}(f) \subset \widetilde{Q} \otimes k(\overline{y})$.  We write $w_{x}(f) \in \operatorname{GW}(k)$ for the isomorphism class of $(Q_{x}(f), \beta_{x})$.
\end{df}

\begin{rmk}\label{rmk_eta_comments}
	We omit the definition of $\widetilde{\eta}$ because it is somewhat involved and we do not make direct use of it.  We do make use of three properties of $\widetilde{\eta}$.  First, the homomorphism has strong base change properties.  Namely, for a noetherian ring $A$ and an $A$-finite quotient $B$ of $A[x_1, \dots, x_n]$ by a regular sequence, Scheja--Storch construct a distinguished $A$-linear function $\eta_{B/A} \colon B \to A$ in a manner that is compatible with extending scalars by an arbitrary homomorphism $A \to \overline{A}$ \cite[page~184, first paragraph]{scheja}.  Second, the pairing $\beta$ is nondegenerate by \cite[Satz~3.3]{scheja}.  Finally, the restriction $\eta_0 \colon Q_{0}(f) \to k$ of $\widetilde{\eta}$ satisfies the condition from Definition~\ref{Definition: EKLform} (by Lemma~\ref{Lemma: EtaSatisfiesEKL} below).  In particular, the definition of $w_{0}(f)$ in Definition~\ref{Definition: EKLform} agrees with the definition of $w_{x}(f)$  from Definition~\ref{Definition: SSForm} when $x=0$.
		
	The reader familiar with \cite{eisenbud77}  may recall that in that paper, where $k=\bbR$, the authors do not make direct use of Scheja--Storch's work but rather work directly with the functional on $\widetilde{Q}$ defined by $a \mapsto \operatorname{Tr}(a/J)$.  Here $\operatorname{Tr}$ is the trace function of the field extension $\operatorname{Frac} \widetilde{Q} / \operatorname{Frac} P_{y}$.  We do not use the function $a \mapsto \operatorname{Tr}(a/J)$ because it is not well-behaved in characteristic $p>0$ since e.g.~the trace can be identically zero.
\end{rmk}

We now describe the properties of the family $(\widetilde{Q}, \widetilde{\beta})$.

\begin{lm} \label{Lemma: EtaSatisfiesEKL}
	The distinguished socle element $E$  satisfies $\eta_{0}(E)=1$.  
\end{lm}
\begin{proof}
	By the construction of $\eta_0$ in \cite{scheja}, $\eta_{0} = \Theta^{-1}(1)$ for a certain explicit $Q_{0}(f)$-linear homomorphism $\Theta \colon \operatorname{Hom}_{k}(Q_{0}(f), k) \to Q_{0}(f)$.  Furthermore, Scheja--Storch prove  $\pi = \Theta^{-1}(E)$ for $\pi \colon Q_{0}(f) \to k$ the evaluation map $\pi(a) = a(0)$ by \cite[proof of (4.7)~Korollar]{scheja}.  By linearity, we have 
	\begin{align*}	
		\Theta(\pi)		=& E \\
					=& E \cdot 1 \\
					=& E \cdot \Theta(\eta_{0}) \\
					=& \Theta( E \cdot \eta_{0}),
	\end{align*}
	so $\pi = E \cdot \eta_{0}$.  Evaluating $\pi(1)$, we deduce 
	\begin{align*}
		1 	=& (E \cdot \eta_{0})(1) \\
			=& \eta_{0}(E). 
	\end{align*}
\end{proof}

\begin{lm} \label{Lemma: FormOrthogonal}
	Let $\widetilde{Q} \otimes k(\overline{y}) = Q_1 \times Q_2$ be a decomposition into a direct sum of rings.  Then $Q_1$ is orthogonal to $Q_2$ with respect to $\widetilde{\beta} \otimes k(\overline{y})$.
\end{lm}
\begin{proof}
	By definition $a_1 \cdot a_2=0$ for all $a_1 \in Q_1$, $a_2 \in Q_2$, hence $(\widetilde{\eta} \otimes k(\overline{y})) (a_1 \cdot a_2)=0$.
\end{proof}

We now prove that, for $\overline{y}_{1}, \overline{y}_{2} \in \bbA^{n}_{k}(k)$, the restriction $\widetilde{\beta}$ to the fiber over $\overline{y}_{1}$ is stably isomorphic to restriction to the fiber over $\overline{y}_2$.  This result follows easily from the following form of Harder's theorem.

\begin{lm}[Harder's theorem] \label{Lemma: HarderLemma}
	Suppose that $( \widetilde{Q}, \widetilde{\beta})$ is a pair consisting of a finite rank, locally free module $\widetilde{Q}$ on $\bbA^1_k$ and a nondegenerate symmetric bilinear form $\widetilde{\beta}$ on $\widetilde{Q}$.  Then $(\widetilde{Q}, \widetilde{\beta}) \otimes k(\overline{y}_{1})$ is stably isomorphic to $(\widetilde{Q}, \widetilde{\beta}) \otimes k(\overline{y}_{2})$ for any $\overline{y}_1, \overline{y}_{2} \in \bbA^1_{k}(k)$.
\end{lm}
\begin{proof}
	When $\operatorname{char} k \ne 2$, the stronger claim that $( \widetilde{Q}, \widetilde{\beta})$ is isomorphic to a symmetric bilinear form defined over $k$ is \cite[Theorem~3.13, Chapter~VII]{lam05b}.  When $\operatorname{char} k = 2$, the claim can be deduced from loc.~cit.~as follows.  By \cite[Remark~3.14, Chapter~VII]{lam05b}, the pair ($\widetilde{Q}, \widetilde{\beta})$ is  isomorphic to an orthogonal sum of a symmetric bilinear form defined over $k$ and a sum of symmetric bilinear forms defined by Gram matrices of the form 
\begin{equation} \label{Eqn: Char2Matrix}
	\begin{pmatrix} a(y) 	& 	1 	\\ 
				1 	&  	0
	\end{pmatrix}
\end{equation}
for $ a(y) \in k[y]$.  This last matrix equals $\mathbb{H}$ when $a(\overline{y})=0$, and otherwise the identity
\[
	\begin{pmatrix}
		a(\overline{y})	&	0	&	0	\\
		0			&	0	&	1	\\
		0			&	1	&	0
	\end{pmatrix}
	=
	\begin{pmatrix}
		1			&	1			&	0	\\
		0			&	1			&	0	\\
		a(\overline{y})	&	a(\overline{y})	&	1
	\end{pmatrix}^{\operatorname{T}}
	\cdot
		\begin{pmatrix}
		a(\overline{y})	&	0			&	0	\\
		0			&	a(\overline{y})	&	1	\\
		0			&	1			&	0
	\end{pmatrix}	\cdot 
	\begin{pmatrix}
		1			&	1			&	0	\\
		0			&	1			&	0	\\
		a(\overline{y})	&	a(\overline{y})	&	1
	\end{pmatrix}
\]
shows that the specialization is stably isomorphic to $\mathbb{H}$.
\end{proof}

\begin{co}  \label{Co: TwoFibersSame}
	The sum 
	\begin{equation} \label{Eqn: ConservedSum}
		\sum_{x \in f^{-1}(\overline{y})} w_{x}(f) \in \operatorname{GW}(k)
	\end{equation}
	is independent of $\overline{y} \in \bbA^{n}_{k}(k)$.
\end{co}
\begin{proof}
	 The sum \eqref{Eqn: ConservedSum} is the class of $\widetilde{\beta} \otimes k(\overline{y})$ by Lemma~\ref{Lemma: FormOrthogonal}, so since any two $k$-points of $\bbA^{n}_{k}$ lie on a line, the result follows from Lemma~\ref{Lemma: HarderLemma}.
\end{proof}	

\begin{rmk} \label{Remark: WhyFiniteAssumption}
	Corollary~\ref{Co: TwoFibersSame} becomes false if the hypothesis that $f$ is finite is weakened to the hypothesis that $f$ is quasi-finite (i.e.~has finite fibers).   Indeed, under this weaker assumption, Scheja--Storch construct a nondegenerate symmetric bilinear form $\beta \otimes k(\overline{y})$ on $\widetilde{Q} \otimes k(\overline{y})$ for every $\overline{y} \in \bbA^{n}_{k}(k)$, but the  class of $\sum w_{x}(f)$ can depend on $\overline{y}$.  For example, consider $k=\bbR$ (the real numbers) and $f := (x_{1}^{3} x_{2}+x_{1}-x_{1}^{3}, x_{2})$.  A computation shows
	\[
			\sum_{x \in f^{-1}(\overline{y})} w_{x}(f)				= \begin{cases}
													\langle 1 \rangle								& \text{ if $\overline{y}_{2}=1$;} \\
													\langle 1/(\overline{y}_{2}-1) \rangle+\mathbb{H} 				& \text{ otherwise,}
												\end{cases}
	\]
	so the rank of $\sum w_{x}(f)$ depends on $\overline{y}$.

The morphism $f$ fails to be finite, and we recover finiteness by passing to the restriction $f \colon f^{-1}(U) \to U$ over $U := \bbA_{k}^2 - \{ y_{2}=1 \}$.  Over $U$, the rank is constant, but the isomorphism class still depends on $\overline{y}$ because 
\[
	\text{signature of $\sum_{x \in f^{-1}(\overline{y})} w_{x}(f)$} = \begin{cases}
											+1	&	\text{ if $\overline{y}_{2}>1$;}\\
											-1	& 	\text{ if $\overline{y}_{2}<1$.}
										\end{cases}
\]
\end{rmk}

We now compute $w_{x}(f)$ when $f$ is \'etale at $x$.

\begin{lm} \label{Lemma: DescentForEtale}
	Let $f \colon \bbA^{n}_{k} \to \bbA^{n}_{k}$ be finite and $\overline{y} \in \bbA^{n}_{k}(k)$ be a $k$-rational point.  If $f$ is \'{e}tale at $x \in f^{-1}(\overline{y})$, then  $$w_{x}(f) = \Tr_{k(x)/k} \langle J(x) \rangle \text{ in $\GW(k)$.}$$ 
	% Let $x$ be a point of $f^{-1}(\overline{y})$ and suppose that $f$ is \'{e}tale at $x$. Then the bilinear form $w_{x}(f)$ is in the isomorphism class of $\Tr_{k(x)/k} \langle J(x) \rangle$, i.e.,

%	Suppose that $f$ is \'{e}tale at $x \in f^{-1}(\overline{y})$ and $L$ is a finite Galois extension of $k$ such that $k(x)$ embeds into $L$.   Let $S$ be the set 
%	\[
%		S = \{  \overline{x} \in \bbA^{n}_{k}(L) \colon \overline{x} (\Spec L)= \{ x\} \},
%	\]
%of $L$-points whose image in $\bbA^n_k$ is $x$. Define $V(S)$ to be the $L$-algebra of functions $S \to L$, with point-wise addition and  multiplication. Define $\phi \colon S \to k$ by $\phi(\overline{x}) = 1/J(\overline{x})$, and $\beta_{\phi} \colon V(S) \times V(S) \to L$ to be the bilinear form
%	\begin{align*}
%		\beta_{\phi}(v_1, v_2)=  \sum_{s \in S} \phi(s) v_{1}(s) \cdot v_{2}(s).
%	\end{align*}The Galois action on $S$ induces an action on $V(S)$, which is descent data. 
	
%	The $k$-bilinear space described by the descent data on the $L$-bilinear space $V(S)$ is isomorphic to the bilinear space $w_{x}(f)=(Q_{x}(f), \beta_x)$.
\end{lm}

Note that when $k(x)=k$, Lemma \ref{Lemma: DescentForEtale} is a consequence of Lemma \ref{Lemma: EtaSatisfiesEKL} and the equality $J = (\operatorname{rank}_{k} Q_{x}(f)) \cdot E = E$ (Remark \ref{rmkJ=rE}). We check the lemma for nontrivial extensions $k \subset k(x)$ using descent.

\begin{proof}
We show that both of these isomorphism classes of bilinear forms over $k$ are described by the following descent data:

Let $L$ be a finite Galois extension of $k$ such that $k(x)$ embeds into $L$, and let $G=\Gal(L/k)$. Let $S$ be the set 
	\[
		S = \{  \overline{x} \in \bbA^{n}_{k}(L) \colon \overline{x} (\Spec L)= \{ x\} \},
	\]
of $L$-points whose image in $\bbA^n_k$ is $x$. Define $V(S)$ to be the $L$-algebra of functions $S \to L$, with point-wise addition and  multiplication. Define $\phi \colon S \to L$ by $\phi(\overline{x}) = 1/J(\overline{x})$, and $\beta_{\phi} \colon V(S) \times V(S) \to L$ to be the bilinear form
	\begin{align*}
		\beta_{\phi}(v_1, v_2)=  \sum_{s \in S} \phi(s) v_{1}(s) \cdot v_{2}(s).
	\end{align*}
Because $J$ is a polynomial with coefficients in $k$, the map $\beta_{\phi}$ is $G$-equivariant. Thus the Galois action on $(V(S),\beta_{\phi})$ determines descent data. 

We now show that the $k$-bilinear space $w_{x}(f)=(Q_{x}(f), \beta_x)$ is isomorphic to the $k$-bilinear space determined by this descent data. To do this, it is sufficient to find a $k$-linear map $Q_{x}(f) \to V(S)$ that respects the bilinear pairings and realizes $Q_{x}(f)$ as the equalizer 
\[
	Q_{x}(f) \to V(S) \rightrightarrows \prod_{\sigma \in G} V(S).
\] There is a tautological inclusion $Q_{x}(f) \to V(S)$ because an element of $Q_{x}(f)$ is a polynomial function on $S$, and we show this inclusion has the desired properties. To see that the inclusion respects the bilinear forms, it suffices to see that the functional $V(\phi)$ restricts to the residue functional  $\eta$. To see the latter, extend scalars to $L$ and then observe that, for every summand $L$ of $Q_{x}(f)$, both $\eta$ and $\phi$ map the Jacobian element to $1$. The equalizer of $V(S) \rightrightarrows \prod_{\sigma \in G} V(S)$ is the subset of $G$-invariant functions (i.e.~functions $v \colon S \to L$ satisfying $v(\sigma s) = v(s)$ for all $\sigma \in G$, $s \in S$).  Because $S$ is finite, every function on $S$ is a polynomial function, and a polynomial function is $G$-invariant if and only if it can be represented by a polynomial with coefficients in $k$, i.e.~lies in $Q_{x}(f)$. Thus, we see that $w_{x}(f)$ is determined by the descent data on $(V(S),\beta_{\phi})$. 

It remains to show that the descent data on $(V(S),\beta_{\phi})$ also determines the $k$-bilinear form $\Tr_{k(x)/k} \langle J(x) \rangle$. The equality $\langle J \rangle = \langle 1/J \rangle$ in the Grothendieck--Witt group shows that $\Tr_{k(x)/k} \langle J(x) \rangle$ has representative bilinear form $B: k(x) \times k(x) \to k$ defined $B(x,y) = \Tr_{k(x)/k}(xy/J)$. The claim is equivalent to the statement that there is a $G$-equivariant isomorphism $ L \otimes_k k(x) \cong V(S)$ respecting the bilinear forms. 

Note that $S$ is in bijective correspondence with the set of embeddings $k(x) \hookrightarrow L$, and that we may therefore view $s$ in $S$ as a map $s: k(x) \to L$. Let $\Theta: L \otimes_k k(x) \to V(S)$ denote the $L$-linear isomorphism defined by $$\Theta (l \otimes q) (s) =l (s q).$$  By definition, $$\beta_{\phi}(\Theta(1 \otimes q_1), \Theta(1 \otimes q_1)) = \sum_{s \in S} (1/J(s)) s(q_1) s(q_2),$$ where $J(s)$ denotes the Jacobian determinant evaluated at the point $s$, and $s(q_i)$ denotes the image of $q_i$ under the embedding $k(x) \to L$ corresponding to $s$. Since $J$ is defined over $k$, we have $J(s) = s(J)$. Thus  $$\sum_{s \in S} (1/J(s)) s(q_1) s(q_2) =  \Tr_{L/k}(q_1 q_2/J) = B(q_1,q_2),$$ showing that $\Theta$ respects the appropriate bilinear forms.
\end{proof}
	
\section{Proof of the Main Theorem} \label{Section: PrOfMainThm}

We first note the case of the Main Theorem when $f$ is \'etale.

\begin{lm} \label{Lemma: etalevalue}
Let $f \colon \bbA^{n}_{k} \to \bbA^{n}_{k}$ be a polynomial function that satisfies Assumption~\ref{Assumption} and $\overline{y} \in \bbA^{n}_{k}(k)$ be a $k$-rational point.  Suppose that $f$ is \'{e}tale at $x \in f^{-1}(\overline{y})$. Then $$w_{x}(f)= \deg^{\bbA^1}_x f.$$
\end{lm}

\begin{proof}
Combine Lemma \ref{Lemma: DescentForEtale} and Proposition \ref{loc_degree_etale_point}. 
\end{proof}

We now use the previous results to  prove the Main Theorem.

\begin{proof}[Proof of Main Theorem]
Recall that, after possibly passing from $k$ to an odd degree field extension $L/k$ when $k$ is a finite field,  we can assume that $f$ satisfies Assumption~\ref{Assumption} by Proposition~\ref{Prop: ReduceToRegSeq_enhanced}.  This allows us to reduce to the case where the assumption is satisfied because the natural homomorphism 
\begin{equation} \label{Eqn: GWExtendScalars}
	\operatorname{GW}(k) \to \operatorname{GW}(L), w \mapsto w \otimes_{k} L
\end{equation}
is injective. Injectivity holds for a somewhat general $L/k$, but we only need the result in the simple case of a finite field in which case the result can be deduced as follows.  When $\operatorname{char} k = 2$, the only invariant of an element of $\operatorname{GW}(k)$ is its rank, so injectivity follows from the observation that extending scalars preserves the rank.  When $\operatorname{char} k \ne 2$, an element of $w \in \operatorname{GW}(k)$ is completely determined by its rank and discriminant.  Thus to show injectivity, we need prove that if $\operatorname{disc}(w \otimes_{k} L) \in (L^{\ast})^{2}$, then $\operatorname{disc}(w) \in (k^{\ast})^{2}$, and this result is e.g.~a consequence of \cite[Corollary~2.6]{lam05}.  

Since the formation of  $w_0(f)$ is compatible with field extensions and similarly with $\deg^{\bbA^1}_{0}(f)$,  we conclude that it is enough to prove the theorem  when $f$ satisfies Assumption~\ref{Assumption}.  

After possibly passing to another odd degree field extension, we can further assume that there exists $\overline{y}_0 \in \bbA^{n}_{k}(k)$ such that $f$ is \'{e}tale at every point of $x \in f^{-1}(\overline{y}_{0})$. (To see this: Since any field extension of degree prime to $p$ is separable, $f$ is \'etale at the generic point \cite[17.6.1c']{egaIV_4}. Therefore $f$ is \'etale when restricted to a Zariski neighborhood of the generic point \cite[D\'efinition 17.3.7]{egaIV_4}. Let $Z$ denote the complement of this open neighborhood. Since finite maps are closed, $f(Z)$ is a closed subset of $\bbA^n_k$ not containing the generic point. Any $k$-valued point of $\bbA_k^n - f(Z)$  has the desired property, and this subset contains $k$-points after possibly passing to an odd degree extension.) 

By these assumptions, $f$ is the restriction of $F \colon \bbP^{n}_{k} \to \bbP^{n}_{k}$, and $F$ induces a map  $\bbP(F): \bbP^n_k/\bbP^{n-1}_k \to  \bbP^n_k/\bbP^{n-1}_k$ of motivic spheres that has degree
\begin{equation}	\label{pmt:local_deg}
	\deg^{\bbA^1}( \bbP(F)) = \sum_{x \in f^{-1}(\overline{y})} \deg^{\bbA^1}_{x}(F) \text{  for any $\overline{y} \in \bbA^{n}_{k}(k)$}
\end{equation}
by the local degree formula (Proposition~\ref{pr_deg_is_sum_local_deg}).  In particular, the right-hand side is independent of $\overline{y}$ since the left-hand side is.

Analogously we proved in  Section~\ref{Section: Family} that the sum
\[
	\sum_{x \in f^{-1}(\overline{y})}w_{x}(f)  \in \operatorname{GW}(k)
\]
 is independent of $\overline{y} \in \bbA^{n}_{k}(k)$ (Corollary~\ref{Co: TwoFibersSame}).

The local terms $w(\beta_{x})$ and $\deg^{\bbA^1}_{x}(F)$ are equal when $F$ is \'{e}tale at $x$ by the Lemma~\ref{Lemma: etalevalue}.  As a consequence 
\begin{equation} \label{Eqn: SumsAreEqual}
	 \sum_{x \in f^{-1}(\overline{y})} \deg^{\bbA^1}_{x}(F)  =  \sum_{x \in f^{-1}(\overline{y})}w_{x}(f)
\end{equation}
for $\overline{y}=\overline{y}_{0}$ and hence (by independence) all $\overline{y} \in \bbA^{n}_{k}(k)$.  In particular, equality holds when $\overline{y}=0$.  By Assumption~\ref{Assumption}, the morphism $F$ is \'{e}tale at every $x \in f^{-1}(0)$ not equal to the origin, so subtracting off these terms from \eqref{Eqn: SumsAreEqual}, we get
\[
	\deg^{\bbA^1}_{0}(f) = w_{0}(f).
\]
\end{proof}

\section{Application to singularity theory} \label{Section: SingThy}
Here we use the local $\bbA^1$-degree to count singularities arithmetically, as proposed in the introduction.  We assume in this section 
\[
	\operatorname{char} k \ne 2,
\]
but see Remark~\ref{Rmk: CharTwo} for a discussion of $\operatorname{char} k =2$.

Specifically, given the equation $f \in P_{x}$ of an isolated hypersurface singularity $X \subset \bbA^{n}_{k}$ at the origin, we interpret the following invariant as a counting invariant:
\begin{df} \label{Def: NodeDefinition}
	If $f \in P_{x}$ is a polynomial such that $\operatorname{grad} f$ has an isolated zero at the closed point $x \in \bbA^{n}_{k}$, then we define $\mu^{\bbA^1}_{x}(f) := \deg^{\bbA^1}_{x}(\operatorname{grad} f)$.  When $x$ is the origin, we write $\mu^{\bbA^1}(f)$ for this class and call it the \textbf{arithmetic Milnor number} or \textbf{$\bbA^1$-Milnor number}.
\end{df}

Two remarks about this definition:
\begin{rmk}
	The condition that $\operatorname{grad} f$ has an isolated zero at  $x$ implies that the fiber of $f$ over $f(x)$ has an isolated singularity at $x$, and the converse is true in characteristic $0$ but not in characteristic $p>0$, as the example of $f(x_1, x_2) = x_1^p+x_2^2$ shows.
\end{rmk}

\begin{rmk}
	When $k=\bbC$, the arithmetic Milnor number is determined by its rank, which is the classical Milnor number $\mu(f) = \rank Q_{0}(f)$.  The classical Milnor number $\mu(f)$ is not only an invariant of the equation $f$ but  in fact is an invariant of the singularity $0 \in \Spec(P_{x}/f)$ defined by $f$.  When $k$ is arbitrary, the invariance properties of $\mu^{\bbA^1}(f)$ are more subtle, especially in characteristic $p>0$.  In particular,  the rank of $\mu^{\bbA^1}(f)$ is not an invariant of the singularity in characteristic $p>0$.  For example, $f(x) = x_1^2 + x_2^p + x_2^{p + 1}$ and $g(x) = x_1^2 + x_2^p + x_2^{2 p +1}$ both define the $A_{p-1}$ singularity  (in the sense that the completed local rings $\widehat{P}_{x}/f$,  $\widehat{P}_{x}/g$, and $\widehat{P}_{x}/x_1^2 + x_2^p$ are isomorphic), but the ranks of $w_{0}(f)$ and $w_{0}(g)$ are respectively $p$ and $2p$. For conditions that imply $\mu^{\bbA^1}(f)$ is an invariant of the singularity, see Lemma~\ref{Lemma: ContactInv}.
\end{rmk}

We now examine the arithmetic Milnor number of a node in more detail.  We define a node following  \cite[Expos\'{e}~XV]{SGA7_10_22}:
\begin{df} \label{Def: Node}
	When $k=\overline{k}$ is algebraically closed, we say that a closed point $x \in X$ of a finite type $k$-scheme is a \textbf{node} (or standard $A_1$-singularity or ordinary quadratic singularity)  if the completed local ring $\widehat{\calO}_{X, x}$ is isomorphic to a $k$-algebra of the form
	\begin{equation} \label{Eqn: StdNode}
		 \widehat{P}/x_1^2+\dots x_n^2 +\text{higher order terms.} 
	\end{equation}
	Here $\widehat{P} = k[[x_1, \dots, x_n]]$ is the power series ring over $k$.
	
	When $k$ is arbitrary, we say that $x \in X$ is a \textbf{node} if every $\widetilde{x} \in X \otimes_{k} \overline{k}$ mapping to $x$ is a node.  We say that $f \in P$ is \textbf{the equation of a node} at a closed point $x \in \bbA^{n}_{k}$ if $x \in \Spec(P/f)$ is a node.  
\end{df}
The $\bbA^1$-Milnor number of the equation of a node is the weight that appears in Equation~\ref{Eqn: XBifurcatesToNodes} from the introduction.  Indeed, if $f = u_1 x_1^2 + \dots + u_n x_n^2$, then $\operatorname{grad}(f) = (2 u_1 x_1, \dots, 2 u_n x_n)$, so
\begin{align*} 
	\mu^{\bbA^1}(f) =&	\left\langle \det \begin{pmatrix} \frac{\partial^2 f}{\partial x_i \partial x_j}(0)  \end{pmatrix} \right\rangle 	\\
	=& \langle 2^{n} u_1 \dots u_n \rangle	\\
	=& \langle  u_1 \dots u_n \rangle	\text{ if $n$ is even.}
\end{align*} 

The arithmetic Milnor number is related to an invariant studied in real enumerative geometry.  Over the real numbers, $1$-nodal curves are typically counted with weights known as Welschinger signs or weights (see e.g.~\cite{Welschinger10}).  Over $k=\bbR$, there are three different types of nodes: the split node (defined by $f=x_1^2-x_2^2$ at the origin), the nonsplit node (defined by $f=x_1^2+x_2^2)$ at the origin), and a complex conjugate pair of nodes.  The Welschinger weights of these nodes are respectively $+1$, $-1$, and $0$. The weight of a real node is exactly the negative of the signature of $\mu^{\bbA^1}(f)$.

Over an arbitrary field, the structure of a node is described by \cite[Expos\'{e}~XV, Th\'{e}or\`{e}me~1.2.6]{SGA7_10_22}.  That theorem states that, if $x \in X := \Spec(P/f)$ is a node, then $L := k(x)$ is a separable extension of $k$ and there exists a nondegenerate quadratic form $q = u_1 x_1^2+\dots+u_n x_n^2 \in L[x_1, \dots, x_n]$ and a morphism $(\Spec(L \otimes_{k} P/q), 0) \to (X, x)$ of pointed $k$-schemes that induces an isomorphism on henselizations.

(Note: in loc.~cit.~the result is stated with $L/k$ the maximal separable subextension of $k(x)/k$, but this subextension is $k(x)/k$ because we have assumed $\operatorname{char} k \ne 2$.)

We can use this description of nodes to describe the arithmetic Milnor number of  a node.
\begin{lm} \label{Lemma: ContactInv}
	Assume  $n$ is even.  Suppose that $L/k$ is a separable field extension and 
	\begin{align*}
		x \in X =& \Spec(L \otimes_{k} P_{x}/f),   \\
		y \in Y =& \Spec(P_{y}/g)
	\end{align*}	
	 are nodes and $(X, x) \to (Y, y)$ is a morphism of pointed $k$-schemes that induces an isomorphism on henselizations.  Then 
	\begin{align} \label{Eqn: ExplicitArithmeticSign}
		\mu^{\bbA^1}_{y}(g) 	=&	\operatorname{Tr}_{L/k}( \left \langle  \frac{\partial^{2} f }{\partial x_i \partial x_j}( x) \right \rangle)			
	\end{align}
\end{lm}
\begin{proof}
	By Proposition~\ref{loc_degree_etale_point}, we have 
	\[		
		\mu^{\bbA^1}_{y}(g) 	=	\operatorname{Tr}_{L/k}( \left \langle  \frac{\partial^{2} g }{\partial x_i \partial x_j}(y) \right \rangle),
	\]
 so it is enough to prove that the Hessian of $f$ differs from the Hessian of $g$ by a perfect square.  Say that $(X, x) \to (Y, y)$ is induced by the ring map defined by   $y_1 \mapsto a_1, \dots, y_n \mapsto a_{n}$.  The elements $a_1, \dots, a_n$ must satisfy 
	\begin{equation} \label{Eqn: ContactEquivalence}
		f= u \cdot g(a_1, \dots, a_n)  \text{ for some $u \in \calO_{X, x}^{h}$}
	\end{equation}
	
	Computing the Hessian of $f$ using \eqref{Eqn: ContactEquivalence}, we deduce
	\[
		\det \begin{pmatrix} \frac{\partial^{2} f}{\partial x_i \partial x_j}(x) \end{pmatrix} = u(x)^n \cdot \det \begin{pmatrix} \frac{\partial^{2} a}{\partial x_i \partial x_j}(x) \end{pmatrix}^2 \cdot  \det \begin{pmatrix} \frac{\partial^{2} g}{\partial y_i \partial y_j}(y) \end{pmatrix}
	\]
	Since $n$ is even, this last equation shows that the two Hessians differ by a perfect square.
\end{proof}
\begin{rmk}
	For $\mu^{\bbA^1}_{x}(f)$ to be an invariant of the pointed $k$-scheme $x \in X$, it is essential that $n$ is even.  For example, when $n$ is odd, consider the equation $f = x_1^2 + \dots x_n^2$ and note that both $f$ and $-f$ define the standard node at the origin, but $\mu^{\bbA^1}(f)= \langle 2^n \rangle$, $\mu^{\bbA^1}(-f) = \langle -2^n \rangle$.  These two classes are equal only when $-1$ is a perfect square. For odd $n$, we get an invariant of the {\em equation} determining the pointed $k$-hypersurface.
\end{rmk}

While $\mu^{\bbA^1}_{x}(f)$ is an invariant of a node when $n$ is even, it does not, in general, determine the isomorphism class, as the following example shows. 

\begin{expl}
	The equations $f=x^2+y^2+z^2+w^2 \in \bbR[w, x, y, z]$ and $g=x^2+y^2-z^2-w^2 \in \bbR[w, x, y, z]$ both define a node at the origin with arithmetic Milnor number $\langle +1 \rangle$. The algebras $\bbR[[w, x, y, z]]/f$ and $\bbR[[w, x, y, z]]/g$ are not isomorphic because~e.g. the projectivized tangent cones are not isomorphic (as $\operatorname{Proj}\bbR[W, X, Y, Z]/  X^2+Y^2-Z^2-W^2$ has infinitely many $\bbR$-points, but $\operatorname{Proj}\bbR[W, X, Y, Z]/ X^2+Y^2+Z^2+W^2$ has none).
\end{expl}

We now  identify $\mu^{\bbA^1}(f)$ as a count of nodes.  Recall that we wish to identify  $\mu^{\bbA^1}_{x}(f)$ with a count of the nodal fibers of the family 
\[
	f(x) -a_1 x_1 - \dots - a_n x_n= t \text{ over the $t$-line}
\]
for $a_1, \dots, a_n \in k$ sufficiently general. In showing this, an essential point is to show that, for $y \in \bbA^{n}_{k}(k)$, the sum of the local degrees $\sum_{f(x)=y} \deg_{y}^{\bbA^1}(\operatorname{grad} f)$ is independent of $x$.  When $\operatorname{grad}(f)$ extends to a suitable morphism $\bbP^{n}_{k} \to \bbP^{n}_{k}$, this is Proposition~\ref{pr_deg_is_sum_local_deg}, but requiring the map to extend is a restrictive condition that fails to be satisfied in important basic examples such as $f = x_1^2+x_2^{n}$.  

We will instead deduce  independence from Corollary~\ref{Co: TwoFibersSame}.  In order to apply that corollary, we need to interpret $\mu^{\bbA^1}_{x}(f)$ in terms of the bilinear pairing $\widetilde{\beta}$.  We have done this when $k(x)=k$ and when $f$ is the equation of a node at $x$ but not in general.  The following lemma is stated so that we only need to consider singularities of this type, allowing  us to avoid a lengthy technical discussion of the relation between $\widetilde{\beta}$ and $\deg^{\bbA^1}_{x}(\operatorname{grad} f)$  when $k(x)$ is a nontrivial extension of $k$.  %This hypothesis is not that restrictive in that, given an $f$, we can modify it along the lines of Proposition~\ref{Prop: ReduceToRegSeq_enhanced} so that the hypothesis is satisfied.

%The following  lemmas allow us to identify the Milnor number in this manner.  

%\begin{lm} \label{Lemma: ConserveDegree}
%	If $g \colon \bbA^n_{k} \to \bbA^{n}_{k}$ is a finite morphism and $\overline{y} \in \bbA^{n}_{k}(k)$, then
%	\begin{equation} \label{Eqn: LocalDegreeConserved}
%		\sum_{x \in g^{-1}(0)} \deg^{\bbA^1}_{x}(g) =  \sum_{x \in g^{-1}(\overline{y})} \deg^{\bbA^1}_{x}(g).
%	\end{equation}
%\end{lm}
%\begin{proof}
%	By the Proposition~\ref{loc_degree_etale_point} and Main Theorem, we have $\deg_{x}^{\bbA^1}(g) = w_{x}(g)$ for all $x$'s appearing in Equation~\eqref{Eqn: LocalDegreeConserved}, and the analogous result for $w_{x}(g)$ is Corollary~\ref{Co: TwoFibersSame}.
%\end{proof}

\begin{lm} \label{Lemma: SingularityCountMostGeneral}
	Let $f \in P$ be such that 
	\[
		\operatorname{grad}(f) \colon \bbA^{n}_{k} \to \bbA^{n}_{k}
	\]
	is a finite morphism. Assume every zero of $\operatorname{grad}(f)$ either has residue field $k$ or is in the \'{e}tale locus of $\operatorname{grad} f$ and similarly with $\operatorname{grad} (f - a_1 x_1 - \dots - a_n x_n)$. Then we have 
	\begin{align}	
		\sum \mu^{\bbA^1}_{x}(f) = &		\sum \mu^{\bbA^1}_{x}(f-a_1 x_1 - \dots -a_n x_n). \label{Eqn: ConstantSum}
	\end{align}
	for any $(a_1, \dots, a_n) \in \bbA^{n}_{k}(k)$.  Here both sums run over all zeros of the relevant gradient.
\end{lm}
\begin{proof}
	Observe that the zeros of $\operatorname{grad}(f - a_1 x_1 - \dots - a_n x_n)$ are exactly the points in the preimage of $(a_1, \dots, a_n)$ under $\operatorname{grad}(f)$.  Furthermore, $\mu^{\bbA^1}_{x}(f-a_1 x_1 - \dots -a_n x_n) = \deg_{x}( \operatorname{grad} f)$.  Thus the left-hand side of \eqref{Eqn: ConstantSum} is the sum of $\deg_{x}^{\bbA^1}(\operatorname{grad} f)$ for $x$ in the preimage of $x$, and the right-hand side is the analogous sum over the preimage of $(a_1, \dots, a_n)$.  By the Proposition~\ref{loc_degree_etale_point} and the Main Theorem, we have $\mu_{x}^{\bbA^1}(f) = w_{x}$, so the result is Corollary~\ref{Co: TwoFibersSame}.
\end{proof}

\begin{lm} \label{Lemma: NodalFamilyExists}
		Let $f \in P$ be such that 
	\[
		\operatorname{grad}(f) \colon \bbA^{n}_{k} \to \bbA^{n}_{k}
	\]
	is a finite, separable morphism.  Then there exists a nonempty Zariski open subset  $U \subset\bbA^{n}_{k}$ such that, for all $(a_1, \dots, a_n) \in \bbA^{n}_{k}$, the preimage of $0$ under
	\begin{gather} \label{Eqn: BifurcatedFamily}
		\operatorname{grad}(f(x) - a_1 x_1 - \dots - a_n x_n) \colon \bbA^{n}_{k} \to \bbA^{n}_{k}
	\end{gather}
	 is \'{e}tale over $k$.
\end{lm}
\begin{proof}
	Observe that since $\operatorname{grad} f$ is separable,  the locus of points  $V \subset \bbA^n_{k}$ where $\operatorname{grad} f$ is \'{e}tale  contains the generic point of $\bbA^n_k$ and hence is a nonempty Zariski open subset.  The subset $\operatorname{grad} f(\bbA^{n}_{k} - V) \subset \bbA^{n}_{k}$ is closed because $\operatorname{grad} f$ is proper and so the complement of $\operatorname{grad} f(\bbA^{n}_{k} - V)$ has  the desired properties.
\end{proof}

\begin{lm} \label{Lemma: RecognizeNode}
Let $f \in P$ is given.  If $f(x) = \operatorname{grad}(f)(x)=0$ and $\operatorname{grad}(f)$ is \'{e}tale at $x$, then  $x \in \Spec(P/f)$ is a node.
\end{lm}
\begin{proof}
	By the definition of a node, we can reduce to the case where $k=\overline{k}$ and, after possibly making a linear change of coordinates, we can assume $x=0$ is the origin.  Write
	$f(x) = \sum a_{\underline{i}} x_1^{i_1} \dots x_{n}^{i_n}$.  Since $f(x) = \operatorname{grad}f(x) =0$, all terms of degree at most $1$ must vanish.  Since $\operatorname{grad}(f)$ is \'{e}tale at $x$, the determinant of the matrix defined by the degree $2$ terms (i.e.~the Hessian) must be nonzero.  We conclude that, after a further linear change of variables (diagonalize the quadratic form), the given equation can be written as
	\[
		f(x) = x_1^2 + \dots +x_n^2 + \text{higher order terms,}
	\]
	showing $x \in \Spec(P/f)$ is a node.
\end{proof}

Combining  the previous  lemmas provides us with the desired interpretation of $\mu^{\bbA^1}(f)$ as a count of nodal fibers.

\begin{co} \label{Corollary: SingularityCount}
	Let $n$ be even and $f \in P$ such that $\operatorname{grad}(f)$ is finite and separable.  Then for $(a_1, \dots, a_n) \in \bbA^{n}_{k}(k)$ a general $k$-point, the family 
	\begin{gather} \label{Eqn: BifurcateFamily}
		\bbA^{n}_{k} \to \bbA^{1}_{k},  \\
		x \mapsto f(x) - a_1 x_1 - \dots a_n x_n 	\nonumber
	\end{gather}
	has only nodal fibers.
	
	Assume every zero of $\operatorname{grad}(f)$ either has residue field $k$ or is in the \'{e}tale locus of $\operatorname{grad} f$.  Then the nodal fibers of \eqref{Eqn: BifurcateFamily} have the  property that:
	\begin{gather} \label{Eqn: SingularityCount}
	\sum \mu^{\bbA^1}_x(f) \\ 			\nonumber
	= \\	 
	\sum \#\left(\text{fibers of \eqref{Eqn: BifurcateFamily} with henselization $0 \in \Spec( \frac{L[x_1, \dots, x_n]}{u_1 x_1^2+\dots+u_n x_n^2})$}\right) \cdot \operatorname{Tr}_{L/k}( \left\langle u_1 \dots u_n \right\rangle) 				\nonumber
	\end{gather} 
	Here we say that $x \in f^{-1}(y)$ has henselization $0 \in \Spec( \frac{L[x_1, \dots, x_n]}{u_1 x_1^2+\dots+u_n x_n^2})$ if the pointed $k$-schemes have isomorphic  henselizations, and we take  the second sum to run over the isomorphism classes of henselizations.% runs over the isomorphism classes of nodes, and $\operatorname{Tr}_{L/k} \colon \operatorname{GW}(L) \to \operatorname{GW}(k)$ is the trace map.
\end{co}
\begin{proof}
	By Lemma~\ref{Lemma: NodalFamilyExists}, a general $k$-point $(a_1, \dots, a_n)$ (i.e.~a $k$-point of nonempty open subscheme of $\bbA^{1}_{k}$) has the property that $\operatorname{grad}(f-a_1 x_1 - \dots -a_n x_n)$ is \'{e}tale at every zero.  We will prove that such a point satisfies the desired conditions.
	
	For this choice of $(a_1, \dots, a_n)$ the family \eqref{Eqn: BifurcateFamily}  has only nodal fibers  by Lemma~\ref{Lemma: RecognizeNode}.  Furthermore, the terms on the right-side of \eqref{Eqn: SingularityCount} are the arithmetic Milnor numbers of the nodal fibers of \eqref{Eqn: BifurcateFamily} by Equation~\eqref{Eqn: ExplicitArithmeticSign}.  Thus that sum is the sum of $\mu_{x}^{\bbA^1}(f - a_1 x_1 - \dots a_n x_n)$ as $x$ runs over the zeros of the gradient, so  \eqref{Eqn: SingularityCount} is a special case of Lemma~\ref{Lemma: SingularityCountMostGeneral}.
\end{proof}

Let us illustrate the content of  Corollary~\ref{Corollary: SingularityCount} with the example of the cusp (or $A_2$) singularity discussed at the end of the introduction.  The polynomial $f = x_1^2+x_2^3$ satisfies the hypotheses of the corollary.  Furthermore the origin is the only zero of $\operatorname{grad}(f)$, and from Table~\ref{Table: ADEsing}, we see that 
$\mu^{\bbA^1}(f) = \mathbb{H}$.  Thus if $(a_1, a_2)$ are chosen so that  \eqref{Eqn: BifurcateFamily} has two $k$-rational nodal fibers  $\{ x_1^2 + u_1 x_2^2=0\}$
and  $\{ x_1^2 + u_2 x_2^2=0\}$, then 
\[
	\mathbb{H} = \langle u_1 \rangle + \langle u_2 \rangle.
\]
Suppose we further specialize to the case $k=\bbQ_5$.  An inspection of discriminants shows that the family cannot contain, for example, the nodes  $\{ x_1^2 + x_2^2 = 0 \}$ and $\{ x_1^2+ 2 \cdot x_2^2=0\}$.  

More generally, there are, up to henselization, four  nodes over $k=\bbQ_5$ with residue field $k$, namely 
\begin{equation} \label{Eqn: 5adicNodes}
	x_1^2 + x_2^2, x_1^2 + 2 \cdot x_2^2, x_1^2 + 5 \cdot x_2^2,\text{ and } x_1^2 + 5 \cdot 2 \cdot x_2^2.
\end{equation}
 If the cusp bifurcates to two of these nodes, then, by Corollary~\ref{Corollary: SingularityCount}, the two nodes must be isomorphic (take the discriminent).  There are no further obstructions:  the family $x \mapsto f(x) + -u^2/3 \cdot x_2$ contains two nodes, each of  which is isomorphic to $\{ x_1^2 + u \cdot x_2^2=0\}$.\hidden{The family $-x_1^2 = g(x_2) = x_2^3 - u^2/3x_2 -t$ has a node when $t$ is chosen so that the solution to $g'(x_2) =  3x_2^2 -u^3/3=0$ is a zero of $g$. Thus there is a node at $x_2 = u/3$, $t=t_0 = -2u^3/27$, $x_1 = 0$. For this value of $t$, the fiber of the family is $-x_1^2 = (x_2 - u/3)^2(x_2+ 2u/3)$, which has the node $-x_1^2 = u(x_2-u/3)^2$.} 
 
 There are also more complicated possibilities for the nodal fibers.   For example, the only singular fiber of  $x \mapsto f(x) + 3 \cdot 5 \cdot x_2$ is the fiber over the closed point $(t^2 + 4 \cdot 5^3)$, a closed point with residue field a nontrivial extension of $k$.\hidden{The family $-x_1^2 = g(x_2) = x_2^3 + 3 \cdot 5 \cdot x_2 -t$ has a node when $t$ is chosen so that the solution to $g'(x_2) =  3x_2^2 + 3 \cdot 5 =0$ is a zero of $g$. Thus there is a node at $x_2^2 = -5$, and $ t= x_2( 2 \cdot 5)$.}  Additional examples describing the collections of nodes that a singularity can bifurcate to can be found in \cite{oberwolffach}.  

\begin{rmk} \label{Rmk: CharTwo}
	We conclude with a remark about the assumption that $\operatorname{char} k  \ne 2$.  When $\operatorname{char} k = 2$, Definition~\ref{Def: NodeDefinition} should not be taken as the definition of a node because  $x_1^2 + \dots + x_n^2 = (x_1+\dots+x_n)^2$ does not define an isolated singularity.  Instead the polynomial $x_1^2+\dots+x_n^2$ should be replaced by
	\[
		f(x) = \begin{cases}
					x_{1}^2 + x_{2} x_{3} + \dots + x_{x-1} x_{n} & \text{$n$ odd;} \\
					x_{1}^{2} + x_{1} x_{2} + \dots + x_{n-1} x_{n} & \text{$n$ even.}
			\end{cases}		
	\]
	Using this last equation, we can define nodes as before, although their classification becomes more complicated (see  \cite[Expos\'{e}~XV]{SGA7_10_22} for details).  
	
	The arithmetic Milnor number can be defined as in odd characteristic, but then it is not a very interesting invariant.  For example, consider the node that is defined by $f(x) = x_1^2 + x_1 x_2 + u x_2^2$ for $u \in k$.  The gradient function $\operatorname{grad} f(x)=(x_2, x_1)$ does not depend on $u$, so $\mu^{\bbA^1}(f)$, and any other invariant obtained from the gradient, does not depend on $u$.  The isomorphism class of the node does depend on $u$: the isomorphism class is classified by the image of $u$ in $k/\{ v^2 + v \colon v \in k \}$.
	
\end{rmk}

\section{Application to cubic surfaces}  \label{Section: CubicSurface}
Here we explain how the $\mathbb{A}^{1}$-degree can be used to arithmetically count the lines on a cubic surface.  In \cite{kass17}, we proved that the lines on a \emph{smooth} cubic surface satisfy
	\begin{equation} \label{Eqn: FinalLineCount}
			\sum_{d \in L^{\ast}/(L^{\ast})^{2}} ( \#\text{lines of type $d$}) \cdot \operatorname{Tr}_{L/k}( \langle d \rangle) 
			= 
				15 \cdot \langle 1 \rangle + 12 \cdot \langle -1 \rangle.
	\end{equation}
The type of a line can be interpreted in several ways, and one interpretation is that it is the local $\mathbb{A}^{1}$-degree (or index) of a global section $\sigma$ of a vector bundle defined using the cubic surface. The global section $\sigma$ has only simple zeros for smooth cubic surfaces, so the main result of this paper is not needed to prove \eqref{Eqn: FinalLineCount}.  The main result can, however, be used to extend that equation to certain singular surfaces, as we now explain.

% The lines on the surface are the zeros of $\sigma$, and the local $\mathbb{A}^{1}$-degree of $\sigma$ at a zero is defined by identifying $\sigma$ with a polynomial function using a trivialization compatible with a given orientation and then defining the local degree of $\sigma$ to be the local degree, in the sense studied of this paper, of the corresponding polynomial function.  For a smooth cubic surface (the case analyzed in loc.~cit.), the global section $\sigma$ has only simple zeros, so the local $\mathbb{A}^{1}$-degree can be computed without making use of the results of this paper.  

The lines on a cubic surface, smooth or not, are always the zeros of a global section $\sigma$.  When the cubic surface is nonruled, the zeros are isolated but, when the surface is singular, possibly nonsimple.  For a nonruled singular cubic surface, Equation~\eqref{Eqn: FinalLineCount} remains valid provided the type is defined to be the local index of $\sigma$.  With this definition, the type of a line can be effectively computed using the main result of this paper.  

%Using an appropriately oriented local trivialization of the relevant vector bundle, $\sigma$ is identified with a function $f \colon \mathbb{A}_{k}^{4} \to \mathbb{A}^{4}_{k}$ near this line, and the type of this line is the local degree of $f$ in the sense of this paper.

For example, consider the cubic surface defined by $x_1^2 x_4+x_2^3+x_3^3$ over a field of characteristic $0$.  This equation is one of the normal forms of a cubic surface with a $D_4$ singularity.  The surface contains the line parameterized by $[S, T] \mapsto [0, -S, S, T]$.  One computes that the type of this line is the local $\mathbb{A}^{1}$-degree at the origin of the polynomial function $f \colon \mathbb{A}_{k}^{4} \to \mathbb{A}^{4}_{k}$ defined by $(a, b, c, d) \mapsto (c^3-3 c^2+3 c, a^2+3 c^2 d-6 c d+3 d, 2 a b+3 c d^2-3 d^2, b^2+d^3)$.  The authors computed the local $\bbA^1$-degree of this function by implementing the method in Table~\ref{Table: HowToCompute} in Mathematica (Version 10.0.2).  With respect to the lexicographical ordering, we have that
\begin{gather*}
	(d^4,c,b d^2,b^2+d^3,a d^2+2 b d,2 a b-3 d^2,a^2+3 d) \text{ is a Gr\"{o}bner basis, }\\
	1,d,d^2,d^3,b,b d,a,a d \text{ is a $k$-basis for $Q_{0}(f)$, and } E = -9 d^3/2.
\end{gather*}
With respect to the exhibited $k$-basis, $w_0(f)$ is represented by the Gram matrix
\[
\left(
\begin{array}{cccccccc}
 0 & 0 & 0 & -\frac{2}{9} & 0 & 0 & 0 & 0 \\
 0 & 0 & -\frac{2}{9} & 0 & 0 & 0 & 0 & 0 \\
 0 & -\frac{2}{9} & 0 & 0 & 0 & 0 & 0 & 0 \\
 -\frac{2}{9} & 0 & 0 & 0 & 0 & 0 & 0 & 0 \\
 0 & 0 & 0 & 0 & \frac{2}{9} & 0 & 0 & -\frac{1}{3} \\
 0 & 0 & 0 & 0 & 0 & 0 & -\frac{1}{3} & 0 \\
 0 & 0 & 0 & 0 & 0 & -\frac{1}{3} & 0 & 0 \\
 0 & 0 & 0 & 0 & -\frac{1}{3} & 0 & 0 & \frac{2}{3} \\
\end{array}
\right).
\]
This class is $\langle 2, 6 \rangle+3 \cdot \mathbb{H}$. (To see this, replace $a d$ with $3 b/2+a d$ in the basis, and the matrix becomes block diagonal.)

\section*{Erratum}

Lemma~\ref{Lemma: WhenFormsEqual} is false as stated. There are two issues. One is that the common element $\phi_1(E) = \phi_2(E)$ should be assumed to be nonzero. The other concerns the case when the characteristic of $k$ is $2$. In that case, the conclusion that $\beta_{\phi_1}$ is isomorphic to $\beta_{\phi_2}$ should be weakened to the statement that these two forms are stably isomorphic. (When the characteristic of $k$ is not $2$, the notions of {\em stably isomorphic} and {\em isomorphic} coincide \cite[Witt's theorem (4.4)]{milnor73}.) Lemma~\ref{Lemma: WhenFormsEqual} is essentially a restatement of \cite[Proposition 3.4 and 3.5]{eisenbud77}, except \cite[Proposition 3.5]{eisenbud77} includes the hypotheses that $\phi_1(E)$ is not zero and the characteristic of $k$ is not $2$. A counterexample in characteristic $2$ is given as follows.

\begin{expl}
Let $k=\bbF_2$. Let $f$ in $k[x]$ be $f=x^2$. Then $Q_0(f) = \frac{k[x]}{(x^2)}$. The socle $E$ is generated by $E_0(f)=x$. Let $\phi_1, \phi_2: Q_0(f) \to k$ be the $k$-linear functions satisfying $\phi_1(x) = \phi_2(x)=1$, $\phi_1(1) = 0$ and $\phi_2(1) = 1$. Then the associated bilinear forms $\beta_{\phi_1}$ and $\beta_{\phi_2}$ have Gram matrices
\[
\beta_{\phi_1} = \begin{pmatrix} 0 & 1 \\ 1 & 0 \end{pmatrix}\quad\quad \beta_{\phi_2} = \begin{pmatrix} 1 & 1 \\ 1 & 0\end{pmatrix}
\] with respect to the $k$-basis $\{1,x\}$ of  $Q_0(f)$. These forms are not isomorphic because $\beta_{\phi_1}(y,y) = 0$ for all $y$ in $Q_0(f)$, but this is not true for $\beta_{\phi_2}$. See \cite[p. 9]{milnor73}.

\end{expl}

%However, $\beta_{\phi_1} $ and $\beta_{\phi_2}$ are {\em stably} isomorphic. This occurs in general. 
Lemma~\ref{Lemma: WhenFormsEqual} is replaced by the following.

\begin{lm}\label{Lemma: WhenFormsEqual-fixed}
%Lemma \ref{Lemma: WhenFormsEqual} is valid when the characteristic of $k$ is not $2$. When the characteristic of $k$ is $2$, the statement is true after changing the conclusion from ``$\beta_{\phi_{1}}$ is isomorphic to $\beta_{\phi_{2}}$" to ``$\beta_{\phi_{1}}$ is stably isomorphic to $\beta_{\phi_{2}}$."

%Lemma \ref{Lemma: WhenFormsEqual} is valid when the characteristic of $k$ is not $2$. When the characteristic of $k$ is $2$

% $\beta_{\phi_{1}}$ is stably isomorphic to $\beta_{\phi_{2}}$.  Furthermore, if $\phi(E) \ne 0$, then  $\beta_{\phi}$ is nondegenerate.
%Suppose the characteristic of $k$ is not $2$. 

If $\phi_1$ and $\phi_2$ are $k$-linear functions satisfying $\phi_1(E) = \phi_2(E) \text{ in $k^*/(k^{\ast})^{2}$}$, then $\beta_{\phi_{1}}$ is stably isomorphic to $\beta_{\phi_{2}}$ and $\beta_{\phi_1}$ and $\beta_{\phi_2}$ are nondegenerate.
\end{lm}

\begin{proof}
We follow \cite[proof of Proposition 3.5]{eisenbud77} with the necessary modifications. Let $Q=Q_0(f)$. By \cite[Proposition 3.1]{eisenbud77}, both $\phi_1$ and $\phi_2$ generate $Q^*$. Thus there exists a unit $u$ of $Q$ such that $\phi_2 = u \phi_1$. Let $\overline{u}$ be the image of $u$ in $Q/\mathfrak{m}_0$. Then \begin{align*} 
\phi_2 (E) &= (u\phi_1)(E) \\
&= \phi_1(uE) \\
&=\phi_1(\overline{u}E)\\
&=\overline{u} \phi_1(E).
\end{align*}
Since $\phi_1(E) = \phi_2(E)$ we have that $\overline{u}=1$. Thus $u = 1 + n$ for $n$ nilpotent.

Define $\ell: Q[t] \to k[t]$ to be the $k[t]$-linear map given by $\ell = \phi_1 \otimes_k k[t]$. Then define $\ell_t : Q[t] \to k[t]$ to be the $k[t]$-linear map given by $$\ell_t = (1+t n)\ell .$$ Let $\beta_t: Q[t] \times Q[t] \to k[t]$ denote the bilinear form $\beta_t(a_1,a_2) := \ell_t(a_1 \cdot a_2)$. Observe that the specializations $t=0$ and $t=1$ are $\beta_{\phi_1}$ and $\beta_{\phi_2}$ respectively. By Harder's theorem (Lemma~\ref{Lemma: HarderLemma}), it thus suffices to show that $\beta_t$ is nondegenerate.

%We claim that $\beta_t$ is non-degenerate. For this it suffices to show that for every $t$ in $\overline{k}$, the bilinear form $\beta_t: Q \otimes_k \overline$

$\beta_t$ can be pulled back along any closed point $t_0:\Spec L \to \Spec k[t]$, producing a bilinear form $\beta_{t_0}: (Q\otimes_k L) \times (Q\otimes_k L) \to L$. It suffices to show that $\beta_{t_0}$ is nondegenerate for all $t_0$. Since $\ell_{t_0}(E) = \phi_1((1+ t_0 n)E) = \phi_1(E) \neq 0$, this follows from  \cite[Proposition 3.4]{eisenbud77}.

%(Although the result is stated in characteristic not $2$, the proof holds)
\end{proof}
	
Thus $\beta_{\phi_{1}}$ and $\beta_{\phi_{2}}$ determine the same element of $\GW(k)$ and the other results of this paper are true as stated. 	

\section{Acknowledgements}
We would like to thank David Eisenbud, Mike Hopkins, Marc Hoyois, Remy van Dobben de Bruyn, and C.~T.~C.~Wall for useful discussions and correspondence. We also thank Joseph Knight, Ashvin Swaminathan, and Dennis Tseng for pointing out the error in Lemma~\ref{Lemma: WhenFormsEqual}.

Jesse Leo Kass was partially sponsored by the National Security Agency under Grant Number H98230-15-1-0264.  The United States Government is authorized to reproduce and distribute reprints notwithstanding any copyright notation herein. This manuscript is submitted for publication with the understanding that the United States Government is authorized to reproduce and distribute reprints.

Kirsten Wickelgren was partially supported by National Science Foundation Award DMS-1406380 and DMS-1552730.

}

%\bibliographystyle{alpha}

%\bibliography{A1Degree}

\begin{thebibliography}{99}

\bibitem[AGZV12]{arnold12}
V.~I. Arnold, S.~M. Gusein-Zade, and A.~N. Varchenko.
\newblock {\em Singularities of differentiable maps. {V}olume 1}.
\newblock Modern Birkh\"auser Classics. Birkh\"auser/Springer, New York, 2012.
\newblock Classification of critical points, caustics and wave fronts,
  Translated from the Russian by Ian Porteous based on a previous translation
  by Mark Reynolds, Reprint of the 1985 edition.

\bibitem[Ayo07]{Ayoub_sixop1}
Joseph Ayoub.
\newblock Les six op\'erations de {G}rothendieck et le formalisme des cycles
  \'evanescents dans le monde motivique. {I}.
\newblock {\em Ast\'erisque}, (314):x+466 pp. (2008), 2007.

\bibitem[BCRS96]{Becker}
E.~Becker, J.~P. Cardinal, M.-F. Roy, and Z.~Szafraniec.
\newblock Multivariate {B}ezoutians, {K}ronecker symbol and {E}isenbud-{L}evine
  formula.
\newblock In {\em Algorithms in algebraic geometry and applications
  ({S}antander, 1994)}, volume 143 of {\em Progr. Math.}, pages 79--104.
  Birkh\"auser, Basel, 1996.

\bibitem[BS11]{boettger}
Simone B{\"o}ttger and Uwe Storch.
\newblock On {E}uler's proof of the fundamental theorem of algebra.
\newblock {\em J. Indian Inst. Sci.}, 91(1):75--101, 2011.

\bibitem[Caz08]{cazanave08}
Christophe Cazanave.
\newblock Classes d'homotopie de fractions rationnelles.
\newblock {\em C. R. Math. Acad. Sci. Paris}, 346(3-4):129--133, 2008.

\bibitem[Caz12]{cazanave}
Christophe Cazanave.
\newblock Algebraic homotopy classes of rational functions.
\newblock {\em Ann. Sci. \'Ec. Norm. Sup\'er. (4)}, 45(4):511--534 (2013),
  2012.

\bibitem[CD12]{CD-triang_cat_mixed_motives}
Denis-Charles Cisinski and Fr\'ed\'eric D\'eglise.
\newblock Triangulated categories of mixed motives.
\newblock {\em Preprint}, available at \url{http://arxiv.org/abs/0912.2110},
  2012.

\bibitem[CLO05]{cox05}
David~A. Cox, John Little, and Donal O'Shea.
\newblock {\em Using algebraic geometry}, volume 185 of {\em Graduate Texts in
  Mathematics}.
\newblock Springer, New York, second edition, 2005.

\bibitem[Eis78]{eisenbud78}
David Eisenbud.
\newblock An algebraic approach to the topological degree of a smooth map.
\newblock {\em Bull. Amer. Math. Soc.}, 84(5):751--764, 1978.

\bibitem[EL77]{eisenbud77}
David Eisenbud and Harold~I. Levine.
\newblock An algebraic formula for the degree of a {$C^{\infty }$} map germ.
\newblock {\em Ann. of Math. (2)}, 106(1):19--44, 1977.
\newblock With an appendix by Bernard Teissier, ``Sur une in{\'e}galit{\'e}
  {\`a} la Minkowski pour les multiplicit{\'e}s''.

\bibitem[GI80]{ivanov}
D.~Ju. Grigor{$'$}ev and N.~V. Ivanov.
\newblock On the {E}isenbud-{L}evine formula over a perfect field.
\newblock {\em Dokl. Akad. Nauk SSSR}, 252(1):24--27, 1980.

\bibitem[Gro65]{egaIV_2}
A.~Grothendieck.
\newblock \'{E}l\'ements de g\'eom\'etrie alg\'ebrique. {IV}. \'{E}tude locale
  des sch\'emas et des morphismes de sch\'emas. {II}.
\newblock {\em Inst. Hautes \'Etudes Sci. Publ. Math.}, (24):231, 1965.

\bibitem[Gro67]{egaIV_4}
A.~Grothendieck.
\newblock \'{E}l\'ements de g\'eom\'etrie alg\'ebrique. {IV}. \'{E}tude locale
  des sch\'emas et des morphismes de sch\'emas {IV}.
\newblock {\em Inst. Hautes \'Etudes Sci. Publ. Math.}, (32):361, 1967.

\bibitem[Hoy14]{Hoyois_lef}
Marc Hoyois.
\newblock A quadratic refinement of the {G}rothendieck-{L}efschetz-{V}erdier
  trace formula.
\newblock {\em Algebr. Geom. Topol.}, 14(6):3603--3658, 2014.

\bibitem[Khi77]{khimshiashvili}
G.~N. Khimshiashvili.
\newblock The local degree of a smooth mapping.
\newblock {\em Sakharth. SSR Mecn. Akad. Moambe}, 85(2):309--312, 1977.

\bibitem[Khi01]{Khimshiashvili01}
G.~Khimshiashvili.
\newblock Signature formulae for topological invariants.
\newblock {\em Proc. A. Razmadze Math. Inst.}, 125:1--121, 2001.

\bibitem[KW16a]{wickelgren16}
Jesse Kass and Kirsten Wickelgren.
\newblock A classical proof that the algebraic homotopy class of a rational
  function is the residue pairing.
\newblock {\em Preprint}, available at \url{http://arxiv.org/abs/1602.08129},
  2016.

\bibitem[KW16b]{oberwolffach}
Jesse~Leo Kass and Kirsten Wickelgren.
\newblock $\mathbb{A}^1$-{M}ilnor number.
\newblock In {\em Oberwolfach Reports}, number 35/2016. European Mathematical
  Society, 2016.

\bibitem[KW17]{kass17}
Jesse~Leo Kass and Kirsten Wickelgren.
\newblock An arithmetic count of the lines on a smooth cubic surface.
\newblock {\em arXiv Preprint} arXiv:1708.01175, 2017.

\bibitem[Lam05a]{lam05}
T.~Y. Lam.
\newblock {\em Introduction to quadratic forms over fields}, volume~67 of {\em
  Graduate Studies in Mathematics}.
\newblock American Mathematical Society, Providence, RI, 2005.

\bibitem[{Lam}05b]{lam05b}
T.~Y. {Lam}.
\newblock {\em {Serre's problem on projective modules.}}
\newblock Berlin: Springer, 2005.

\bibitem[Lev17]{levine17}
Marc Levine.
\newblock Toward an enumerative geometry with quadratic forms.
\newblock {\em arXiv Preprint} arXiv:1703.03049, 2017.

\bibitem[Mat89]{Matsumura_CRT}
Hideyuki Matsumura.
\newblock {\em Commutative ring theory}, volume~8 of {\em Cambridge Studies in
  Advanced Mathematics}.
\newblock Cambridge University Press, Cambridge, second edition, 1989.
\newblock Translated from the Japanese by M. Reid.

\bibitem[Mil68]{milnor68}
John Milnor.
\newblock {\em Singular points of complex hypersurfaces}.
\newblock Annals of Mathematics Studies, No. 61. Princeton University Press,
  Princeton, N.J.; University of Tokyo Press, Tokyo, 1968.
  
\bibitem[MH73]{milnor73}
John Milnor and Dale Husemoller.
\newblock {\em Symmetric bilinear forms}.
\newblock Springer-Verlag, New York-Heidelberg, 1973.
\newblock Ergebnisse der Mathematik und ihrer Grenzgebiete, Band 73.

\bibitem[Mor04]{Morel_motivicpi0_sphere}
Fabien Morel.
\newblock On the motivic {$\pi_0$} of the sphere spectrum.
\newblock In {\em Axiomatic, enriched and motivic homotopy theory}, volume 131
  of {\em NATO Sci. Ser. II Math. Phys. Chem.}, pages 219--260. Kluwer Acad.
  Publ., Dordrecht, 2004.

\bibitem[Mor06]{morel06}
Fabien Morel.
\newblock {$\Bbb A^1$}-algebraic topology.
\newblock In {\em International {C}ongress of {M}athematicians. {V}ol. {II}},
  pages 1035--1059. Eur. Math. Soc., Z\"urich, 2006.

\bibitem[MV99]{morelvoevodsky1998}
F.~Morel and V.~Voevodsky.
\newblock {${\mathbb A}^1$}-homotopy theory of schemes.
\newblock {\em Inst. Hautes \'Etudes Sci. Publ. Math.}, (90):45--143 (2001),
  1999.

\bibitem[Pal67]{palamodov}
V.~P. Palamodov.
\newblock The multiplicity of a holomorphic transformation.
\newblock {\em Funkcional. Anal. i Prilo\v zen}, 1(3):54--65, 1967.

\bibitem[SGA73]{SGA7_10_22}
{\em Groupes de monodromie en g\'eom\'etrie alg\'ebrique. Expos\'es X \`a
  XXII}.
\newblock Lecture Notes in Mathematics, Vol.~340. Springer-Verlag, Berlin,
  1973.
\newblock S{\'e}minaire de G{\'e}om{\'e}trie Alg{\'e}brique du Bois-Marie
  1967--1969 (SGA 7II), Dirig{\'e} par P.~Deligne et N.~Katz.

\bibitem[SS75]{scheja}
G{\"u}nter Scheja and Uwe Storch.
\newblock \"{U}ber {S}purfunktionen bei vollst\"andigen {D}urchschnitten.
\newblock {\em J. Reine Angew. Math.}, 278/279:174--190, 1975.

\bibitem[Voe03]{Voevodsky_MCZ2}
Vladimir Voevodsky.
\newblock Motivic cohomology with {${\bf Z}/2$}-coefficients.
\newblock {\em Publ. Math. Inst. Hautes \'Etudes Sci.}, (98):59--104, 2003.

\bibitem[Wal83]{wall83}
C.~T.~C. Wall.
\newblock Topological invariance of the {M}ilnor number mod {$2$}.
\newblock {\em Topology}, 22(3):345--350, 1983.

\bibitem[Wel10]{Welschinger10}
Jean-Yves Welschinger.
\newblock Invariants entiers en g\'eom\'etrie \'enum\'erative r\'eelle.
\newblock In {\em Proceedings of the {I}nternational {C}ongress of
  {M}athematicians. {V}olume {II}}, pages 652--678. Hindustan Book Agency, New
  Delhi, 2010.

\end{thebibliography}

\end{document}